\documentclass[12pt]{amsart}
\usepackage{geometry}
\usepackage{amssymb}
\usepackage{latexsym}
\usepackage{amsfonts}
\usepackage{amsmath}
\usepackage{amsthm}
\usepackage{xypic}
\usepackage{graphicx}
\usepackage{mathptmx}
\usepackage{url}

\newtheorem{thm}{Theorem}[section]
\newtheorem{conj}[thm]{Conjecture}
\newtheorem{lemma}[thm]{Lemma}
\newtheorem{prop}[thm]{Proposition}
\newtheorem{cor}[thm]{Corollary}
\newtheorem{definition}{Definition}[section]

\newtheorem{eg}[thm]{Example}
\newtheorem{notation}[definition]{Notation}

\def\Z{{\mathbb Z}}
\def\Q{{\mathbb Q}}

\def\pone{{\mathbb P}^1}
\def\ratd{{\mathrm{Rat}_d}}

\def\pgltwo{{\mathrm{PGL}_2}}

\def\OC{{\mathcal{O}_C}}

\def\OCp{{\mathcal{O}_{C,p}}}
\def\mmp2{{\mathfrak{m}}_p / {\mathfrak{m}}_p^2}

\def\Ratd{\textrm{Rat}_d}

\def\Pk{\mathbb{P}_k^{2d+1}}
\def\a{\mathbf{a}}
\def\b{\mathbf{b}}
\def\c{\mathbf{c}}

\def\PP{\mathbb{P}}

\def\div{\mathrm{div}}

\def\Pic{\mathrm{Pic}}
\def\Div{\mathrm{Div}}
\def\GL{\mathrm{GL}}
\def\PGL{\mathrm{PGL}}

\def\Rat{\mathrm{Rat}}

\def\Specmax{\mathrm{Specmax}}

\def\Ecal{\mathcal{E}}
\def\Mcal{\mathcal{M}}
\def\Ocal{\mathcal{O}}

\def\Cfrak{\mathfrak{f}}
\def\Ufrak{\mathfrak{U}}

\title[Resultant and Conductor]{Resultant and Conductor of Geometrically Semi-stable Self Maps of the Projective Line Over a Number Field or Function Field}

\subjclass[2010]{37P05, 14L24, 11G05, 14G99, 37P45}
\date{\today}
\thanks{The three authors recieved funding from the NSF Grants DMS-0854746 and DMS-0739346. The first author was also partially funded by PSC Cuny grant EG 36276.}
\keywords{resultant, conductor, discriminant, bad reduction, critical bad reduction, Latt\`{e}s map, minimality,self map of the projective line}

\author{Lucien Szpiro}
\address{Lucien Szpiro\\
        Ph.D. Program in Mathematics\\
        CUNY Graduate Center\\
        365 Fifth Avenue, New York, NY 10016-4309 U.S.A.}
\email{lszpiro@gc.cuny.edu}

\author{Michael Tepper}
\address{Michael Tepper\\
        Division of Science and Engineering\\
        Penn State Abington\\
        1600 Woodland Road\\
        Abington, PA 19001 U.S.A.}
\email{mlt16@psu.edu}

\author{Phillip Williams}
\address{Phillip Williams\\
        Ph.D. Program in Mathematics\\
        CUNY Graduate Center\\
        365 Fifth Avenue, New York, NY 10016-4309 U.S.A.}
\email{philwill@gmail.com}

\begin{document}

\maketitle

\begin{abstract}
We study the minimal resultant divisor of self-maps of the projective line over a number field or a function field and its relation to the conductor. The guiding focus is the exploration of a dynamical analog to Theorem \ref{szpirostheorem}, which bounds the degree of the minimal discriminant of an elliptic surface in terms of the conductor. The main theorems  of this paper (\ref{degreetwominimal} and  \ref{ssminimaldegreetwo}) establish that, for a degree 2 map, semi-stability in the Geometric Invariant Theory sense on the space of self maps, implies minimality of the resultant.  We prove the ÔsingularÕ reduction of a semi-stable presentation coincides with the simple bad reduction (Proposition \ref{reductionlocus}).  Given an elliptic curve over a function field with semi-stable bad reduction, we show the associated Latt\`{e}s map has unstable bad reduction (Example \ref{unstable}).  Degree $2$ maps in normal form with semi-stable bad reduction are used to construct a counterexample (Example \ref{dynamicalszpirofalse}) to a simple dynamical analog to Theorem \ref{szpirostheorem}. Finally, we consider the notion of ``critical bad reduction,'' and show that a dynamical analog to Theorem \ref{szpirostheorem}, using the locus of critical bad reduction to define the conductor (Example \ref{alltogether}, Proposition \ref{critnumberfield}), is reasonable.\end{abstract}

\setcounter{tocdepth}{2}
\renewcommand{\contentsname}{Content}
\tableofcontents

\section{Introduction}

\subsection{The minimal resultant and semi-stability}
Let $k$ be an algebraically closed field of characteristic $0$. In 1978 the first author proved a theorem which bounds the minimal discriminant of a semi-stable elliptic surface in terms of the conductor \cite{Szpiro1} (Lemma 3.2.2 and Proposition 4.2), \cite{Szpiro3}.
\begin{thm}
\label{szpirostheorem}
Let $f: \Ecal \rightarrow C$ be a proper and flat morphism of a projective surface $\Ecal$, smooth over $k$, to a curve $C$, projective, smooth, of genus q and geometrically connected over $k$. Suppose that the generic fiber of $f$ is an elliptic curve, $E$, smooth and geometrically connected over the function field of $C$. Suppose, further, that $f$ is not isotrivial and the degenerate fibers are semi-stable. Then, if $\Delta_{\Ecal}$ is the discriminant divisor of $\Ecal$ and if $s$ is the number of geometric points of $C$ where the fibers are not smooth, one has:
$$deg (\Delta_{\Ecal}) \leq 6(2q-2+s).$$
\end{thm}

Theorem~\ref{szpirostheorem}, as originally stated, actually gives a bound for the characteristic $p$ case as well. The fact that the bad fibers are assumed to be semi-stable implies that the discriminant divisor, bounded by the theorem, is minimal, in the sense defined below. In \cite{Szpiro2}, the result is strengthened to include the case where reduction is not assumed to be semi-stable.

Associated to any elliptic curve is the Latt\`es self map of $\mathbb{P}^1$. This is obtained by looking at the $x$ coordinate of a Weierstrass equation for the curve under the multiplication by 2 (or, more generally, multiplication by $n$) endomorphism induced by the group structure of the elliptic curve. Motivated by this connection, in this paper we explore a ``dynamical analogue'' to Theorem \ref{szpirostheorem}.  We can formulate what it means for a dynamical system to have bad reduction, and we can construct an associated divisor derived from this (a conductor).  For dynamical systems there is a natural analog to the discriminant, called the resultant, which also tells us whether a dynamical system has bad reduction. Then we can ask whether there is a bound on the ratio of the degree of the \textit{minimal} resultant divisor to the degree of the conductor divisor. (In the number field case degree is replaced by norm for an effective divisor.)

In this paper we analyze in detail Latt\`es maps and maps of degree 2 on $\mathbb{P}^1$. The main results are Theorems \ref{ssminimaldegreetwo} (the function field case) and \ref{degreetwominimal} (the number field case) which assert that semi-stable reduction into the space parametrizing self maps of degree 2 for the action of $\mathrm{PGL}_2$ implies minimality of the resultant.

Other notable results in this paper are:
 
\begin{itemize}
\item The Latt\`es map associated to a semi-stable elliptic curve (i.e an elliptic curve with multiplicative reduction) is  never semi-stable in the space of maps of degree 4 for the action of $\mathrm{PGL}_2$ (Example \ref{unstable})
\item If the conductor is defined as the support of the "simple bad reduction" (i.e where the self map has a lower degree in reduction) then we show by constructing a counterexample that an inequality of the sort mentioned above is not possible (Example \ref{dynamicalszpirofalse}).
\item If one uses the more refined notion of critical conductor (Definition \ref{criticalconductor}) the counterexample just mentioned ceases to be a counterexample in the function field case. Moreover, in the number field case, the question of whether the counterexample still applies can can be reduced to the original conjecture of the first author about the discriminant of elliptic curves.
\end{itemize}

\newpage

\subsection{Setup and Notation}\label{RationalAndResultant}

The general context for the rest of the paper is as follows: we are considering dynamical systems of the projective line over a field $K$ which is the stalk at the generic point of a noetherian integral one dimensional scheme $C$ whose local rings are discrete valuation rings. Sometimes, by abuse of language, we will call these schemes curves. Throughout, we may also assume for simplicity that $K$ is characteristic $0$. Two fields of particular interest are $K$ a number field (here $C=\mathrm{Spec}(\mathcal{O}_K)$) and $K = k(C)$ a function field of an non-singular curve (possibly projective) over an algebraically closed field $k$ of characteristic $0$. By $p \in C$, we always mean a closed point of the scheme (unless otherwise specified), and for $p\in C$, $\kappa(p)$ denotes the residue field at $p$.


\section{Background}

\subsection{Parameterizing Morphisms of $\PP^1$}\label{ParameterizingMorphisms}
Let $F$ be a field, and let $\varphi : \pone_F \rightarrow \pone_F$ be a rational map.  Since the target is a projective curve, every rational map on $\pone_F$ is, in fact, a morphism. Thus, if we choose coordinates $[X,Y]$ for $\pone_F$, $\varphi$ is given by two homogeneous polynomials of the same degree, subject to the restriction that the polynomials do not vanish at a common point of $\pone_{\bar{F}}$. Such a choice of coordinates determines what we will call a \textbf{presentation} of $\varphi$. 
\begin{definition} Let $\varphi$ be a morphism of degree $d$. A \textbf{presentation} $\Phi$ of $\varphi$, with respect to the choice of coordinates $[X,Y]$, is the point $[\a,\b] = [a_0, ..., a_d, b_0, ..., b_d] \in \mathbb{P}^{2d+1}(F)$  given by the coefficients of a pair of homogeneous polynomials $F_\mathbf{a} = a_o X^d + a_1X^{d-1}Y + ... + a_d Y^d$, $F_\mathbf{b} = b_o X^d + b_1X^{d-1}Y + ... + b_d Y^d$ defining the morphism $\varphi$.
\end{definition}
\begin{definition} Let $A$ be a ring with fraction field $K$ and let $\Phi$ be a presentation of a morphism $\varphi$ over $K$.  A \textbf{model} of $\Phi$ is an affine representative $(\a, \b)$ for the projective point $\Phi = [\a, \b]$. A \textbf{model of} $\Phi$ \textbf{over} $A$ is a model of $\Phi$ with coefficients in $A$. If $p \in C$, a $p$-model of $\varphi$ is a model $(\a,\b)$ of $\varphi$ over $\OCp$, such that at least one of the coordinates of $(\a,\b)$ is a unit.
\end{definition}
The \textbf{resultant} is a polynomial constructed from the coefficients of two polynomials, which we will denote by $\rho$. Basic information on the resultant can be found in \cite{SilvermanDynamics}. The condition that two degree $d$ polynomials over $F$ share no common zero over $\pone(\bar{F})$ is equivalent to the non-vanishing of their resultant.  The resultant $\rho$ is homogeneous of degree $2d$ in the coefficients $a_0, ... , a_d, b_o, ... , b_d$. In fact, $\rho$ is bi-homogeneous of degree $d$ in each of $a_0, ... , a_d$ and $b_o, ... , b_d$.  It has coefficents in $\Z$. One can define the resultant as the determinant of the $2d \times 2d$ matrix:

$$\left(\begin{array}{ccccccccc}a_0 & a_1 & \cdots & a_{d-1} & a_d & 0 & \cdots & \cdots & 0 \\0 & a_0 & a_1 & \cdots & a_{d-1} & a_d & 0 & \cdots & 0 \\ &  & \ddots & \ddots & \ddots & \ddots & \ddots &  &  \\ &  & \ddots & \ddots & \ddots & \ddots & \ddots &  &  \\0 & \cdots & \cdots & a_0 & a_1 & \cdots & \cdots & a_{d-1} & a_d \\b_0 & b_1 & \cdots & b_{d-1} & b_d & 0 & \cdots & \cdots & 0 \\0 & b_0 & b_1 & \cdots & b_{d-1} & b_d & 0 & \cdots & 0 \\ &  & \ddots & \ddots & \ddots & \ddots &  &  &  \\ &  & \ddots & \ddots & \ddots & \ddots & \ddots &  &  \\0 & \cdots & \cdots & b_0 & b_1 & \cdots & \cdots & b_{d-1} & b_d\end{array}\right)$$

\vspace{.25cm}
Silverman has shown that presentations of morphisms of degree $d$ over $F$, for a fixed choice of coordinates $[X,Y]$, are in one to one correspondence with $F$ valued points of an affine variety defined over $\Z$:
\begin{notation}$\mathrm{Rat}_d  := \mathbb{P}^{2d+1} \backslash V(\rho)$. This affine variety is isomorphic to $\mathrm{Spec}(R)$, where $$R= \Z[A_0, ... , A_d, B_o, ... , B_d]_{(\rho)}$$ is the subring of elements of degree zero in the localization $\Z[A_0, ... , A_d, B_o, ... , B_d]_{\rho}$.
\end{notation}

In fact, Silverman shows more generally that the scheme $\Rat_d$ gives a universal family parameterizing rational maps over any base scheme.  See \cite{SilvermanSpace} for details.


\subsection{The group action and the quotient scheme}
Also important for our purposes will be the conjugation action of $\mathrm{PGL}_2(F)$ on morphisms. We can describe this action in terms of what it does to the coefficients of the parameterization described above; in fact we have a group action not just on $\mathrm{Rat}_d(F)$ but on the entire projective space $\mathbb{P}^{2d+1}(F)$.

If $\Gamma =\left(\begin{array}{cc}\alpha & \beta \\\gamma & \delta \end{array}\right)\in\GL_2(F)$ and $(X,Y)$ are coordinates for $\mathbb{A}_F^2$, define $(X,Y) \cdot \Gamma$ to be $(\alpha X + \beta Y, \gamma X + \delta Y)$.

If $\Phi = [\a, \b]$ is a presentation of $\varphi$, define a $\GL_2(F)$ action on the model $(\a,\b)$ by sending it to the new the coefficients $(\a^{\Gamma}, \b^{\Gamma})$ obtained from the following pair of polynomials:
$$(F_\mathbf{a}( (X,Y) \cdot \Gamma), F_\mathbf{b}((X,Y) \cdot \Gamma)) \cdot \Gamma^{\mathrm{adj}},$$
where $\Gamma^{\mathrm{adj}}=\left(\begin{array}{cc}\delta & -\beta \\-\gamma & \alpha \end{array}\right)$.  This is actually a group action on all of $\mathbb{A}^{2d+2}(F)$. And it descends to a well defined group action when passing to $PGL_2(F)$ and the projective space $\mathbb{P}^{2d+1}(F)$. We denote this action by $\Phi^{\Gamma} = [\a^{\Gamma}, \b^{\Gamma}]$. It sends $\mathrm{Rat}_d$ to itself, and is thus a group action on morphisms (it corresponds to the usual conjugation action, with respect to the basis $(X,Y)$). This follows from the following formula for how the resultant form transforms under the $\GL_2(F)$ action:
\begin{align}
\label{ResGrpAction}
\rho(\a^{\Gamma},\b^{\Gamma}) = (\mathrm{det}(\Gamma))^{d^2+d} \rho(\a,\b).\end{align}
Thus the non-vanishing of the resultant is preserved by the group action.

When a group acts on a scheme, geometric invariant theory (GIT) gives conditions under which there is a good notion of a quotient scheme for this action. The machinary of GIT is developed in \cite{Dolgachev} and \cite{MFK}. In \cite{SilvermanSpace}, Silverman describes how it is applied to $\ratd$. He constructs  a quotient scheme $\Mcal_d$ and a natural map $\ratd \rightarrow \Mcal_d$. The scheme $\Mcal_d$ is affine and can be described explicitly as the spectrum of the ring of functions invariant under the group action.

For technical reasons, one must work instead with the special linear group $SL_2$ in these constructions, instead of $PGL_2$. We will not go into details in regards to this issue, but what we will need for our purposes is the following: if $k$ is an algebraically closed field, then the orbits of $SL_2(k)$ and $PGL_2(k)$ on $\mathbb{P}^{2d+1}(k)$ are identical, and the points $\Mcal_d(k)$ are in one to one correspondence with the orbit of points in $\Ratd(k)$.

GIT gives two open subsets of $\mathbb{P}^{2d+1}$ containing $\ratd$ of special interest,
$$\ratd \subseteq (\mathbb{P}^{2d+1})^s \subseteq (\mathbb{P}^{2d+1})^{ss}.$$ These are called the stable locus and semi-stable locus, respectively. Using GIT, one can define a \textit{geometric quotient} $(\Mcal_d)^s$ on the former, and a \textit{categorical quotient} $(\Mcal_d)^{ss}$ on the latter, and this gives rise to a sequence of inclusions,
$$\Mcal_d \subseteq (\Mcal_d)^s \subseteq (\Mcal_d)^{ss}.$$

There is a general numerical criterion, given by GIT, for when a point is stable or semi-stable.  For algebraically closed fields, Silverman in \cite{SilvermanDynamics} and \cite{SilvermanSpace} has also worked out what this means for the coefficients.
\begin{thm}\label{numericalcriterion}
Let $k$ be algebraically closed.
\begin{quote}
{\bf (1)}  A point of $\PP^{2d+1}(k)$ is not in $(\PP^{2d+1})^{ss}(k)$ if and only if, after a $SL_2(k)$ conjugation, it satisfies $$a_i = 0\ \textrm{ for all }\ i \leq \frac{d-1}{2}, \textrm{ and } \ b_j = 0\ \textrm{ for all }\ j \leq \frac{d+1}{2}.$$
{\bf (2)}  A point of $\PP^{2d+1}(k)$ is not in $(\PP^{2d+1})^s(k)$ if and only if, after an $SL_2(k)$ conjugation, it satisfies $$a_i = 0\ \textrm{ for all }\ i < \frac{d-1}{2}, \textrm{ and } \ b_j = 0\ \textrm{ for all }\ j < \frac{d+1}{2}.$$
\end{quote}
\end{thm}

There is a reformulation of this numerical criterion that is sometimes more intuitive for applications. We first came across the statement of this criterion in \cite{Ressayre}, though it seems well known.
\begin{thm}
\label{enhancednumericalcriterion}
Let $k$ be algebraically closed.

Let $ [\a, \b] \in \PP^{2d+1}(k)$. Let $\psi$ be the rational map on $\pone_k$ obtained by cancelation of the greatest common factor of $F_\mathbf{a}$ and $F_\mathbf{b}$. Note that $\psi$ has degree $d-D$, where $D$ is the degree of the greatest common factor.
\begin{quote}
{\bf (1)}  Suppose $d=2r$ is even. Then  $[\mathbf{a},\mathbf{b}]$ is unstable if and only if it is not stable, which happens if and only if either

(a) $F_\mathbf{a}$ and $F_\mathbf{b}$ have a common root in $\pone_k$ of order $r+1$ or

(b) $F_\mathbf{a}$ and $F_\mathbf{b}$ have a common root in $\pone_k$ of order $r$ which is also a fixed point of $\psi$.\\
{\bf (2)}  Suppose $d=2r+1$ is odd. Then $[\mathbf{a}, \mathbf{b}]$ is unstable if and only if either

(a) $F_\mathbf{a}$ and $F_\mathbf{b}$ have a common root in $\pone_k$ of order $r+2$ or

(b) $F_\mathbf{a}$ and $F_\mathbf{b}$ have a common root in $\pone_k$ of order $r+1$ which is also a fixed point of $\psi$.\\
Meanwhile, $[\mathbf{a}, \mathbf{b}]$ is not stable if and only if either

(a') $F_\mathbf{a}$ and $F_\mathbf{b}$ have a common root in $\pone_k$ of order $r+1$ or

(b') $F_\mathbf{a}$ and $F_\mathbf{b}$ have a common root in $\pone_k$ of order $r$ which is also a fixed point of $\psi$.
\end{quote}
\end{thm}

\begin{proof}

\textit{(sketch)}

We will sketch the idea for the case where $d$ is even. The other cases are similar.

In the even case, the stable locus coincides with the semi-stable locus, and Theorem \ref{numericalcriterion} says that a point $[\a,\b]$ is outside of this set if and only if one of its conjugates $[\a',\b']$, under the $PGL_2(k)$ action, has the coefficients $a_o', \cdots a_{r-1}'$, $b_0', \cdots b_{r}'$ vanishing. This happens if and only if $Y$ is a common root of $F_{\a'}, F_{\b'}$ of order at least $r$. It is easy to see that $a_r'$ is nonzero also if and only the common root $Y$ is of order exactly $r$ and $[1,0]$ is a fixed point of the map obtained by cancelation. Checking that the group action preserves these properties (but changes the common root and fixed point in question), one can verify that statement (1) follows.
\end{proof}

Points which fall outside of $(\PP^{2d+1})^s$ are sometimes called \textbf{not stable} and those which fall outside of $(\PP^{2d+1})^{ss}$ are called  \textbf{unstable}. We will sometimes work with presentations defined over non-algebraically closed fields (for example, the residue fields of the ring of integers of a number field). Stability or semi-stability in this case is determined by considering these points over the algebraic closure of the field in question and applying Theorem \ref{numericalcriterion} or Theorem \ref{enhancednumericalcriterion}. 

\subsection{Singular and Bad Reduction of Rational maps}

Let us consider a morphism $\varphi$ and a presentation $\Phi \in \Rat_d(K)$. For each $p$, we may obtain a point in $\PP^{2d+1}({\kappa(p)})$ in a natural way, as follows.
\begin{notation}
Let $(\a,\b)$ be a model of $\Phi$. Set $c_i = a_i$ for $0 \leq i \leq d$ and $c_i = b_{i-d-1}$ for $d+1 \leq i \leq 2d+1$, and set $(\c) = (c_0, ... ,c_{2d+1})$. Let $p \in C$. Set \begin{align}n_p(\a, \b) & := \mathrm{min}_j{(v_p(c_j))} .\end{align}
\end{notation}

If $(\a,\b)$ is a $p$-model of $\Phi$, then clearly $n_p = 0$. Thus for any model $(\a,\b)$ of $\Phi$, if we chose a scalar $u_p$ such that $v_p(u_p) = -n_p$, then $(u_p \a, u_p \b)$ is a $p$-model.
\begin{definition}
\label{reduction}
Let $p \in C$ and let $(\a, \b)$ be a model of $\Phi$ over $\OCp$ such that $n_p(\a,\b) = 0$ (i.e. a p-model). \textbf{The reduction of $\Phi$ at $p$} is the point $\Phi_p = [\a(p), \b(p)] \in \PP^{2d+1}({\kappa(p)})$ obtained by evaluating each coefficient of $(\a,\b)$ at $p$ (i.e. looking at the image of that coefficient in the residue field). 
\end{definition}
It is easy to verify that this definition does not depend on the choice of the model. For most points, the reduction of $\Phi$ at $p$ will be a presentation of a morphism over $\kappa(p)$ of degree $d$. However, for a finite set of points, this will fail: the reduced coefficients will describe two polynomials with non-trivial common zeros, and thus will not describe a morphism of degree $d$. (We can, however, cancel the common zeros and obtain a morphism of lower degree). Where this happens is captured by evaluating the resultant form of an appropriate model.

More precisely, given any model $(\a, \b)$ and the choices above, by definition:
\begin{equation}
v_p(\rho(u_p \a, u_p \b)) = v_p( \rho ( \a, \b)) - 2d n_p(\a , \b).
\end{equation}

It can be verified that this number depends only on the presentation $[\a,\b]$.
\begin{notation}
Let $\Phi = [\a,\b]$ be a presentation of $\varphi$. Set: $$N_{\Phi, p}:=v_p( \rho ( \a, \b)) - 2d n_p(\a , \b)).$$
\end{notation}
\begin{definition}
Let $\varphi:\PP^1_K\rightarrow\PP^1_K$ and let $\Phi=[\a,\b]$ be a presentation of $\varphi$. The \textbf{resultant divisor of the presentation} $\Phi$ is
$$R_{\Phi} = R_{[\a,\b]} = \sum_{p \in C}{N_{\Phi,p}}[p].$$
We refer to the support of $R_{\Phi}$ as the \textbf{singular reduction locus of $\Phi$}, and if $p$ is in this support, we say that $\Phi$ has \textbf{singular reduction} at $p$. If $p$ is not in the support of $R_{\Phi}$, then we say $\Phi$ has \textbf{non-singular reduction} at $p$.
\end{definition}

Another way to define the resultant divisor of a presentation is the following: for each $p \in C$, the reduction of $\Phi$ at $p$ is an element of $\PP^{2d+1}(\kappa(p))$ which may or may not be in $\Ratd (\kappa(p))$. In fact, we may construct a morphism of schemes,
$$F_{\Phi}: C \rightarrow \PP^{2d+1}$$
for which when we compose with the natural inclusion morphism $\kappa(p) \hookrightarrow C$ the resulting point of $\PP^{2d+1}(\kappa(p))$ is precisely this reduction of $\Phi$ at $p$. The resultant divisor of the presentation $\Phi$ is then the divisor of zeros of the section of $F_{\Phi}^{*}O(2d)$ given by pulling back the resultant form.


\section{Minimality}
\subsection{The Minimal Resultant and the Conductor}

Given a morphism $\varphi$, the resultant divisors of two different presentations may be different. We may, in particular, be able to act in such a way that singular reduction becomes non-singular reduction. Therefore it is useful to look instead for a notion of a resultant divisor which is invariant under the $\pgltwo(K)$ action, and thus depends only on the morphism $\varphi$. To obtain this, we consult the discussion in \cite{SilvermanDynamics}, section 4.11, which develops the \textit{minimal} resultant in the number field case. The definition of the minimal resultant, as well as Proposition~\ref{valuationformulas} and Proposition~\ref{wclass}, are taken from this discussion.

\begin{definition}
\label{minimalresultant}
Let $\varphi$ be a morphism and $\Phi = [\a,\b]$ a presentation of $\varphi$. Let $$\epsilon_p(\varphi) = \underset{\Gamma \in PGL_2(K)}{\textrm{min}}N_{\Phi^{\Gamma}, p}.$$ The \textbf{minimal resultant} of $\varphi$ is
$$R_{\varphi} = \underset{p \in C}{\sum} \epsilon_p(\varphi)[p].$$
\end{definition}

This divisor is invariant under the $PGL_2(K)$ action, and has support on the points of $C$ for which every $PGL_2(K)$ conjugate of $\Phi$ (that is, every presentation of $\varphi$) has singular reduction.

\begin{definition}
\label{conductor}
Let $\varphi$ be a morphism. The \textbf{conductor} of $\varphi$ is the divisor: $$\Cfrak_{\varphi} = \sum_{p\in\textrm{Support}(R_\varphi)}[p].$$
If $p$ has a nonzero coefficient in $\Cfrak_{\varphi}$, then we say $\varphi$ has \textbf{bad reduction} at $p$. If the coefficient of $p$ in $\Cfrak_{\varphi}$ is zero, then we say $\varphi$ has \textbf{good reduction} at $p$.
\end{definition}
\begin{notation}
Let $D$ be a divisor on $C$. Then $(D)_p$ is the coefficient of $D$ at $p$.
\end{notation}

A dynamical analog to Theorem \ref{szpirostheorem} would bound the degree of the minimal resultant in terms of the degree of the conductor. Without additional assumptions, however, the dynamical analog to Theorem \ref{szpirostheorem} is not always true, at least for degree $2$ maps, as we see in the following theorem.
\begin{eg}
\label{dynamicalszpirofalse}
Let $K = k(t)$. For each $N \in \Z^{+}$, let $\varphi$ be the degree 2 morphism given by: $$\varphi = \frac{X^2 + \lambda_1XY} {\lambda_2XY + Y^2}$$ where $\lambda_1 = a+bt^N$, $\lambda_2 = a^{-1} + b't^N$, $a, b, b' \neq 0,1$, and $ab'+b/a = 0$. Then the degree of the conductor of $\varphi$ is at most 2, and the degree of the minimal resultant is at least $2N$.
\end{eg}
\begin{proof}
This follows immediately from Corollary \ref{counterexample0} and Corollary \ref{counterexample} in Section 4.
\end{proof}

There is a similar example for the number field case, in Example \ref{counterexample0numberfield} below.

\subsection{Conjugation}The difficulty in calculating $R_{\varphi}$ lies in understanding, for a given $p$, which presentation $\Phi$ of $\varphi$ truly realizes the minimal value for $N_{\Phi,p}$.

Let now $\Gamma \in GL_2(K)$, $\Gamma = \left(\begin{array}{cc}\alpha & \beta \\\gamma & \delta\end{array}\right)$. As above, we write $(\a^\Gamma, \b^\Gamma)$ for the new coefficients under the action of $\Gamma$ on a model of $\Phi$.  Recall that this is a group action that descends to the conjugation action under the projection to $PGL_2$ and projective space of the coefficients. Applying valuations to formula (\ref{ResGrpAction}) above, we have:
\begin{equation}\label{vpResGrpAction}
v_p(\rho (\a^\Gamma, \b^\Gamma)) = v_p(\rho( \a , \b )) + (d^2 + d) v_p(\mathrm{det}(\Gamma)).
\end{equation}
This tells us what the action does to the valuation of the resultant form. We can also say something about what happens to $n_p(\a, \b)$ under the same action. Let
$$v_p(\Gamma) = \mathrm{min}(v_p(\alpha), v_p(\beta), v_p(\gamma), v_p(\delta)).$$
Then we have an inequality
\begin{equation}\label{npGrpAction}
n_p(\a^\Gamma, \b^\Gamma)  \geq n_p(\a, \b) + (d+1)v_p(\Gamma).
\end{equation}
To see this, observe that each coefficient in $(\a^\Gamma, \b^\Gamma)$ is a sum of terms of the form,
$$(\textrm{coefficient of } (\a, \b)) \cdot ( \textrm{homogeneous polynomial of degree } d + 1 \textrm{ in }\Z[\alpha, \beta, \gamma, \delta]).$$

If we write $c_i^{\Gamma}$ for the $i$-th coefficient  of $(\a^\Gamma, \b^\Gamma)$, for each $i$, we have some polynomial $f$ and a $j$ such that,
\begin{align*} v_p(c_i^{\Gamma}) & \geq v_p(c_j \cdot f(\alpha, \beta, \gamma, \delta)) \\
 & \geq n_p(\a, \b) + (d+1)v_p(\Gamma).
 \end{align*}

 Another useful observation is the following. Suppose $\Gamma_p \in GL_2(\OCp)$. Then
 \begin{equation}\label{vpResEq}
 v_p( \rho(\a, \b)) = v_p(\rho(\a^{\Gamma_p}, \b^{\Gamma_p}))
 \end{equation}
and
\begin{equation}\label{npEq}
n_p(\a, \b) = n_p(\a^{\Gamma_p}, \b^{\Gamma_p}).
\end{equation}
These quickly follow from the above formulas when it is observed that being in $GL_2(\OCp)$ is equivalent to having $v_p(\textrm{det}(\Gamma_p)) = 0$ and $v_p(\Gamma_p) \geq 0$. Thus we see immediately (\ref{vpResEq}) follows from (\ref{vpResGrpAction}).
For (\ref{npEq}), $v_p(\Gamma_p) \geq 0$ implies $n_p (\a^{\Gamma_p}, \b^{\Gamma_p}) \geq n_p(\a,\b)$ by (\ref{npGrpAction}). But since $\Gamma_p ^{-1} \in GL_2(\OCp)$ as well, $n_p (\a, \b) \geq n_p (\a^{\Gamma_p}, \b^{\Gamma_p})$.

Notice also that (\ref{vpResEq}) and (\ref{npEq}) together imply:
\begin{align*}
N_{\Phi, p} &= N_{\Phi^{\Gamma},p}.
\end{align*}
The above formulas are some basic tools for understanding how the valuation of the resultant changes under the action of $K^\times$ and the actions of $PGL_2(K)$ and $GL_2(K)$ on presentations $[\a,\b]$ and models $(\a, \b)$ respectively. Let us collect them in the following proposition.
\begin{prop}
\label{valuationformulas}
Let $\varphi$ be a morphism, let $\Phi$ be a presentation of $\varphi$, and let $(\a,\b)$ be a model of $\Phi$. If $\Gamma \in \mathrm{GL}_2(K)$, then
\begin{quote}
{\bf (1)}  $v_p(\rho (\a^\Gamma, \b^\Gamma)) = v_p(\rho( \a , \b )) + (d^2 + d) v_p(\mathrm{det}(\Gamma))$\\
{\bf (2)}  $n_p(\a^\Gamma, \b^\Gamma)  \geq n_p(\a, \b) + (d+1)v_p(\Gamma)$\\
{\bf (3)}  If $\Gamma_p \in \mathrm{GL}_2(\mathcal{O}_{C,p})$ then $v_p( \rho(\a, \b)) = v_p(\rho(\a^{\Gamma_p}, \b^{\Gamma_p}))$, \\
{\bf (4)}  $n_p(\a, \b) = n_p(\a^{\Gamma_p}, \b^{\Gamma_p})$ and\\
{\bf (5)}  $N_{\Phi,p} = N_{\Phi^{\Gamma_p},p}$.
\end{quote}
\end{prop}


\subsection{Minimal Presentations and Models}
Is it possible to find a global model or presentation that realizes the minimal resultant at each point? If $C$ is affine, it is reasonable to look for a \textit{model} that does this; for example, in the case of a number field $K$, Silverman asks if there is a $\Gamma$ such that $\Phi^\Gamma$ can be written as $[\a',\b']$ where $(\a',\b')$ has coordinates in the ring of integers $\Ocal_K$ of the number field, and where \begin{align}
\label{NumberFieldMinimalModel}
R_{\varphi} = \underset{\mathfrak{p} \in \Specmax(\Ocal_K)}{\sum} v_p(\rho (\a,\b))[\mathfrak{p}].\end{align}
This is an analogy of a global minimal model of an elliptic curve over a number field. This is discussed in \cite{SilvermanEllipticCurves} (Chapter VIII, Section 8).  More generally, if $C$ is affine, such a definition makes sense. However, if $C$ is a complete nonsingular curve, if we could even find a presentation and a model that gave us the value of $R_{\varphi}$ for all points of $C$, it would follow that $R_{\varphi}$ is trivial, being an effective divisor that is also principal. Likewise, requiring the coefficients to be in the global sections of $C$ would imply triviality. So this definition is not very useful in this case. In regards to the number field situation, Silverman mentions a notion in \cite{SilvermanDynamics} (Exercise 4.46b) that will be useful, which we may formulate in our setting as follows.
\begin{definition} Let $U$ be an open subset of $C$ and let $S$ be the complement of $U$. An \textbf{$S$-minimal global model} of $\varphi$ is a model $(\a,\b)$ over $\OC(U)$ of a presentation $\Phi$ of $\varphi$ such that $(R_{\varphi})_p = v_p(\rho(\a,\b))$ for every $p$ in $U$.
\end{definition}

In addition to a minimal model, we might look for a minimal presentation.
\begin{definition} We say that a presentation $\Phi = [\a,\b]$ is a \textbf{minimal presentation} of $\varphi$ if: $$ R_{\varphi} = \underset{p \in C}{\sum} N_{\Phi, p}[p].$$
\end{definition}

In \cite{SilvermanDynamics}(Proposition 4.99), Silverman proves a necessary condition for having a minimal model in the sense of (\ref{NumberFieldMinimalModel}) for number fields. Here we show that this condition works in our setting, when instead we consider $S$-minimal models.
\begin{prop}
\label{wclass}
Let $\Phi \in \ratd(K)$ be a presentation of $\varphi$.
\begin{quote}
{\bf (1)} If $d$ is odd, then for each model $(\a, \b)$ of $\Phi$, there is a divisor $A_{(\a,\b)}$ satisfying,
$$R_{\varphi} = \textrm{div}(\rho(\a,\b)) + 2d A_{(\a, \b)}.$$
{\bf (2)} If $d$ is even, then for each model $(\a, \b)$ of $\Phi$, there is a divisor $A_{(\a,\b)}$  satisfying,
 $$R_{\varphi} = \textrm{div}(\rho(\a,\b)) + d A_{(\a, \b)}.$$
 \end{quote}
The image of this divisor in $Pic(C)$ is independent of both the chosen model, and of the chosen presentation $\Phi$. In the even case, the divisors associated to two different models of $\varphi$ differ by a principal divisor coming from the square of an element in $K$.
\end{prop}

\begin{proof}
For each $p \in C$, there is some $\Gamma_p \in GL_2(K)$ such that coefficient of $R_{\varphi}$ at $p$ is given by:\begin{align*}v_p(\rho(\a^{\Gamma_p}, \b^{\Gamma_p}))- 2dn_p(\a^{\Gamma_p}, \b^{\Gamma_p}) &= v_p(\rho(\a, \b)) + (d^2 + d)v_p(\textrm{det}(\Gamma_p)) - 2dn_p(\a^{\Gamma_p}, \b^{\Gamma_p}).
 \end{align*}
In the odd case, we can factor $2d$ out of the terms on the right. In the even case, we can factor $d$ out. Thus in the odd case we simply define,
$$A_{(\a, \b)} = \underset{p}{\sum} [\frac{(d+1)}{2}v_p(det(\Gamma_p)) - n_p(\a^{\Gamma_p}, \b^{\Gamma_p})][p].$$
And in the even case we define,
$$A_{(\a, \b)} = \underset{p}{\sum} [(d+1)v_p(det(\Gamma_p)) -2 n_p(\a^{\Gamma_p}, \b^{\Gamma_p})][p].$$
Now recall that $R_{\varphi}$ is invariant under the group action. So keeping in mind how the valuation of the resultant changes with respect to the $\GL_2$ action (Proposition~\ref{valuationformulas}), and with respect to scalar multiplication, it is a calculation to see that the divisor class is invariant under both scalar multiplication and the $PGL_2(K)$ action. For the even case, let $M \in \mathrm{GL}_2(K)$. We have
\begin{align*}
\mathrm{div}(\rho(\a,\b)) + dA_{(\a, \b)} &= R_{\varphi} \\
&= \mathrm{div}(\rho(\a^M,\b^M)) + dA_{(\a^M, \b^M)} \\
&= \mathrm{div}(\rho(\a,\b)) + (d^2+d)\mathrm{det}(M) + dA_{(\a^M, \b^M)}.
\end{align*}
Then
\begin{align*}
A_{(\a, \b)} &= \mathrm{div}(\mathrm{det}(M)^{(d+1)})+ A_{(\a^M, \b^M)}.
\end{align*}
A similar calculation can be made for scalar multiplication of the coefficients and the odd case.
\end{proof}
\begin{cor}
Suppose $\varphi$ has an $S$-minimal global model. Then the image of $A_{(\a,\b)}$ in the restricted group $Pic(U)$ (where $U = C - S$) is trivial.
\end{cor}
\begin{proof}
By assumption there exists some $\Gamma$ that allows us to write
$$R_{\varphi}|_U = \div(\rho(\a',\b'))|_U$$
where $[\a', \b'] = [\a^{\Gamma}, \b^{\Gamma}]$. Then $A_{(\a', \b')}|_U = 0$ in $\mathrm{Div}(U)$.  So its class in $\mathrm{Pic}(U)$ is also 0.  However, this equals the image of $A_{(\a, \b)}|_U$ in $\mathrm{Pic}(U)$, by the above and basic properties of the restriction homomorphism on $\Div$ and $\Pic$.
\end{proof}

\begin{cor}
Suppose $d$ is odd and $U \subset C$ is open. Then if $(\a, \b)$ is a model of $\Phi$ and $A_{(\a,\b)}$ is trivial in $Pic(U)$, then there is $c \in K$ such that $v_p(\rho(c\a, c\b))=(R_{\varphi})_p$ for all $p \in U$. In particular, if $U = Spec(A)$ with $A$ a principal ideal domain, then such $c$ always exists.
\end{cor}
\begin{proof}
 By assumption, we can write $A_{(\a, \b)}|_U = \div(c)|_U$, where $c \in K$. Then
 \begin{align*} R_{\varphi}|_U & = (\div(\rho(\a, \b)) + (2d)\div(c))|_U
 \\ &=\ \div(\rho(c\a, c\b)|_U.
 \end{align*}
\end{proof}
Notice that this is not saying that the resultant divisor of the presentation $[\a,\b]$ has the value of $R_{\varphi}$ for each $p \in U$. For that to happen, we would have to be able to evaluate $(c\a, c\b)$ at each $p \in U$. We don't know that we can do this, because some of these coefficients may have poles in $U$.


\subsection{Local Minimality Conditions}
Even locally, finding the value of the minimal resultant is difficult, since we have to take into consideration all conjugates of a given map. One thing we can say initially is the following.
\begin{prop}
Suppose $U = Spec(A)$ is a PID and $\varphi$ has a presentation $\Phi = [\a,\b]$ where the model $(\a, \b)$ is over $\OC(U)$, and the coefficients have no common factors (We can always do this in a PID). Let $p \in U$. Then if $v_p(\rho(\a, \b)) <  2d$, then $v_p(\rho(\a, \b)) = (R_{\varphi})_p$ at $p$.
\end{prop}
\begin{proof}
Just write $R_{\varphi}|_U = div(\rho(c\a,c\b))$ as in the above proof.  The conditions on the coefficients imply that $N_{\Phi,p} = v_p(\rho(\a,\b))$. Thus if $v_p(c) > 0$, we contradict the minimality of $R_{\varphi}$, and if $v_p(c) < 0$, we contradict the fact that $R_{\varphi}$ is an effective divisor.
\end{proof}

Let $p \in C$. Any $[\Gamma] \in PGL_2(K)$ is a presentation of an associated degree $1$ map, and we may always take our $\Gamma \in GL_2(K)$ to be $p$-model of $[\Gamma]$. This is the same as requiring that $v_p(\Gamma) = 0$, as defined above. Under this assumption, the valuation at $p$ of the determinant of $[\Gamma]$ is well defined. In fact, the determinant is just the degree $1$ case of the resultant, and so this is just the value at $p$ of the resultant divisor of the presentation $[\Gamma]$ of the associated degree one map. We will denote this number by $v_p(\mathrm{det}[\Gamma])$.
\begin{prop}
\label{minimalitycondition}
Let $p \in C$. The presentation $\Phi$ with $p$-model $(\a,\b)$ realizes the minimal resultant at $p$ if and only if for every $[\Gamma] \in PGL_2(K)$ such that $v_p(\mathrm{det}[\Gamma])) > 0$, we have 
\begin{align}\label{minimalityconditionineq}
\frac{n_p(\a^{\Gamma}, \b^{\Gamma})}{v_p(\mathrm{det}[\Gamma])} \leq \frac{d+1}{2}
\end{align}
where $(\a, \b)$ is a $p$-model of $\Phi$.
\end{prop}
\begin{proof}
The assumption that $(\a,\b)$ is a $p$-model means $n_p(\a,\b) = 0$. Thus, recalling the definition of $R_{[\a,\b]}$ we have that
\begin{align*}
(R_{[\a,\b]})_p \leq (R_{[\a^{\Gamma}, \b^{\Gamma}]})_p \\
& \iff \\
& v_p(\rho(\a,\b))  \leq v_p(\rho(\a, \b)) + (d^2+d)v_p(\mathrm{det}[\Gamma]) - 2dn_p(\a^{\Gamma}, \b^{\Gamma})
\end{align*}
and the statement follows by canceling and dividing.
\end{proof}
To check the condition (\ref{minimalityconditionineq}), it may be helpful to know that it suffices to check it for only certain sorts of conjugates. In the function field case, we have the following:
\begin{prop}
\label{neededcases}
Let $K = k(C)$ where $C$ is a nonsingular projective curve over $k$. To check minimality at $p$, it suffices to check that the condition of Proposition~\ref{minimalitycondition} holds for $[\Gamma] \in PGL_2(K)$ where $\Gamma = \left(\begin{array}{cc}\alpha & \beta \\\gamma & \delta\end{array}\right)$ is a $p$-model of one of the following $3$ forms:
\begin{quote}
{\bf (1)}  $\alpha(p), \beta(p) = 0; \gamma(p), \delta(p) \neq 0$\\
{\bf (2)}  $\alpha(p), \beta(p) \neq 0; \gamma(p), \delta(p) = 0$\\
{\bf (3)}  $\alpha(p), \beta(p), \gamma(p), \delta(p) \neq 0.$
\end{quote}
\end{prop}
\begin{proof}
First note that by Proposition~\ref{valuationformulas} (5), any conjugation by a matrix in $\mathrm{GL}_2(k)$ will have no affect on the value of $R_{[\a,\b]}$. Thus if $M \in \mathrm{GL}_2(k)$, we have
$$(R_{[\a^{\Gamma},\b^{\Gamma}]})_p = (R_{[\a^{\Gamma \cdot M},\b^{\Gamma \cdot M]}})_p.$$
So it is no loss of generality to multiply the $\Gamma$ we start with \textit{on the right} by such a matrix. By the same proposition, we need only consider conjugations for which
$$v_p(\mathrm{det}(\Gamma)) > 0.$$
The possible configurations for the vanishing of $\alpha, \beta, \gamma, \delta$ that yield this are
$$A=\left(\begin{array}{cc} &  \\ & \end{array}\right),\
B=\left(\begin{array}{cc} & * \\ * & *\end{array}\right),\
C=\left(\begin{array}{cc} * & \\ * & *\end{array}\right),\
D=\left(\begin{array}{cc} * & * \\ * & \end{array}\right),\
E=\left(\begin{array}{cc} * & * \\ & *\end{array}\right),$$
$$F=\left(\begin{array}{cc} * & \\ * &\end{array}\right),\
H=\left(\begin{array}{cc} & * \\ & *\end{array}\right),\
J=\left(\begin{array}{cc} * & * \\ &\end{array}\right),\ \textrm{and }
L=\left(\begin{array}{cc} & \\ * & *\end{array}\right)$$
where a $*$ indicates that the coefficient has positive valuation.  Then since,
$$\left(\begin{array}{cc}\alpha & \beta \\
\gamma & \delta \end{array}\right) \left(\begin{array}{cc}1 & 1 \\
0 & 1 \end{array}\right) = \left(\begin{array}{cc}\alpha & \alpha + \beta \\
\gamma & \gamma + \delta\end{array}\right),$$
we see by the basic addition inequalities of valuations that we can reduce the case of $B$ into case of $L$, the case of $E$ into the case of $J$, and the case of $H$ into the case of $A$. Likewise, we have that
$$\left(\begin{array}{cc}\alpha & \beta \\
\gamma & \delta\end{array}\right) \left(\begin{array}{cc}1 & 0 \\
1 & 1\end{array}\right) = \left(\begin{array}{cc}\alpha + \beta & \beta \\
\gamma + \delta & \delta\end{array}\right).$$
This reduces the case of $C$ to $L$, $D$ to $J$, and $F$ to $A$. Thus only the cases of $A$, $J$, and $L$ remain.
\end{proof}

To reduce the forms in cases $J$ and $L$ to the form in case $A$, we could multiply on the left by some $M$ in an analogous way as above. However since we are acting first by $M$ instead of by $\Gamma$, it is not immediate as before that $(R_{[\a^{\Gamma},\b^{\Gamma}]})_p = (R_{[\a^{M \cdot \Gamma},\b^{M \cdot \Gamma}]})_p$.

In Section 4, we will prove another minimality criterion having to do with the symmetric functions of the multipliers of periodic points.


\section{Semi-stability}

\subsection{Semi-stable Presentations}
One of our initial inspirations, to formulate a dynamical analog to Theorem~\ref{szpirostheorem}, comes from the natural map
$$C \rightarrow \PP^{2d+1}$$
used in \cite{PST}, and mentioned above, that arises from a given morphism $\varphi:\PP^1_K\rightarrow\PP^1_K$  and a presentation $\Phi$. Let us outline precisely how this map is obtained. A presentation of a morphism $\Phi$ is a $K$-valued point of the scheme $\Ratd$, i.e. it is a morphism $K \rightarrow \Ratd$. Composing with the inclusion into projective space, we get a map $K \rightarrow \PP^{2d+1}$. Let $U$ be any affine open subset of $C$. Then the valuative criterion for properness tells us that this map now extends uniquely to give us a map $U\rightarrow\PP^{2d+1}$. By the uniqueness given by the valuative criteria, these maps must agree on intersections, and so we get a unique map $C\rightarrow\PP^{2d+1}$.

If this map lands in the semi-stable locus $(\PP^{2d+1})^{ss}$, then we say that the presentation $\Phi$ is a semi-stable presentation:
\begin{definition}
Let $\Phi$ be a presentation of $\varphi$. If the morphism $C\rightarrow\PP^{2d+1}$ factors through $(\PP^{2d+1})^{ss}$, then $\Phi$ is a \textbf{semi-stable presentation} of $\varphi$.
\end{definition}
If $\Phi$ is a semi-stable presentation, then for all $p \in C$, the reduction of $\Phi$ at $p$ is in $(\PP^{2d+1})^{ss}(\kappa(p))$. The following theorem says that if $\Phi$ is a semi-stable presentation of $\varphi$, then the locus of non-singular points for $\Phi$ is equal to the locus of points of good reduction of $\varphi$.
\begin{thm}\label{reductionlocus}
Let $\varphi:\PP^1_K\rightarrow\PP^1_K$ be a morphism with degree $d>1$ and $\mathfrak{f}_\varphi$ be the conductor of $\varphi$.  Let $\Ufrak_\varphi=C\backslash \mathfrak{f}_\varphi$ be the set of points of $C$ where $\varphi$ has good reduction and $\Phi=[\a,\b]$ be a semi-stable presentation of $\varphi$.  Let $U_\Phi$ be the open set of $C$ such that $\Phi$ has non-singular reduction.  Then $U_\Phi=\Ufrak_\varphi$.
\end{thm}
\begin{proof}
First we will show, without any loss in generality, we may replace the field $K$ with a finite extension $K'$.  That is, it is sufficient to prove the theorem for the map induced by a finite extension $K'$, $\varphi_{K'}:\PP^1_{K'}\rightarrow\PP^1_{K'}$.  Considering this extension is necessary to use techniques from \cite{PST}. Under the assumptions on $C$, one can construct a curve $C'$ and a finite morphism $C' \rightarrow C$. This is obtained by patching together the integral closures in $K'$ of an affine open cover of $C$. 

Suppose $U_{\Phi_{K'}}=\Ufrak_{\varphi_{K'}}$, where $U_{\Phi_{K'}}$ and $\Ufrak_{\varphi_{K'}}$ are the subsets of $C'$ defined above associated to the morphism $\varphi_{K'}$ and presentation $\Phi_{K'}$ defined over the extension field $K'$.  Suppose the locus of singular reduction $T'=C'\backslash U_{\Phi_{K'}}$ and the conductor $\mathfrak{f}_{\varphi_{K'}}=C'\backslash \Ufrak_{\varphi_{K'}}$ are equal.  If we restrict the finite map $C'\rightarrow C$ to $T'\subseteq C'$, we obtain a surjection $\varpi:T'\rightarrow T$, where $T=C\backslash U_\Phi$.  If $p\in T$ then there exists a $p'\in T'$ such that $\varpi(p')=p$.  Since $T'=\mathfrak{f}_{\varphi_{K'}}$ it follows that $p=\varpi(p')\in\mathfrak{f}_\varphi$, because the conductor over $K'$ will land in the conductor over $K$ via the map $C'\rightarrow C$.  This shows, $T\subseteq\mathfrak{f}_\varphi$.  We have $\mathfrak{f}_\varphi\subseteq T$ by construction.  Therefore $\mathfrak{f}_\varphi= T$ and $U_\Phi=\Ufrak_\varphi$.

Now, by the discussion above there exists a map $U_\Phi\rightarrow\Rat_d$, and composing with the quotient we have
$$U_\Phi\rightarrow\Rat_d\rightarrow\Mcal_d.$$
The first half of this can be trivially extended to $C\rightarrow\PP^{2d+1}$.  Given that $\Phi$ is a semi-stable presentation we know that this extension factors through the semi-stable locus $(\PP^{2d+1})^{ss}$.  Therefore composing with the quotient map we obtain
$$C\stackrel{F_\Phi}{\longrightarrow} (\PP^{2d+1})^{ss}\stackrel{\pi}{\longrightarrow}(\Mcal_d)^{ss}$$
and denote this composition by $f:C\rightarrow(\Mcal_d)^{ss}$.

We will now be applying the methods in \cite{PST}; to do so we will have to consider the ``good reduction loci" of models of presentations, rather than of the presentations themselves, as the argument there requires. So for any presentation $\Psi$ and associated model $(\a,\b)$ we define $U_{(\a,\b)}$ to be the open subset of $C$ consisting of those points $p$ where $(\a,\b)$ is a $p$-model over $\OCp$ and the resultant $\rho(\a,\b)$ is a unit. This is a subset of the non-singular reduction locus of $\Psi$, and also yields a map $U_{(\a,\b)}\rightarrow\Rat_d$.  Composing with the quotient map as above gives
$$U_{(\a,\b)}\rightarrow\Rat_d\rightarrow\Mcal_d.$$
Considering now all possible models of presentations of $\varphi$, we note that each open set $U_{(\a,\b)}$ is contained in $\Ufrak_\varphi$ and we can take a collection of these open sets $\{U_{(\a,\b)_i}\}$ to create a cover of $\Ufrak_\varphi$. Choose a cover that contains $U = U_{(\a_{0},\b_{0})}$ where now $(\a_0,\b_0)$ is a model of the presentation $\Phi$. Using the arguments from \cite{PST} \S 3.4 we can construct a morphism
$$\Ufrak_\varphi\rightarrow\Mcal_d$$
from the morphisms $U_{(\a,\b)_i}\rightarrow\Rat_d\rightarrow\Mcal_d$.  Due to the methods in \cite{PST} we have to consider a base extension of $K$.  However, we have shown that this will not affect our argument.  Now we have the composition
$$\Ufrak_\varphi\rightarrow\Mcal_d\hookrightarrow(\Mcal_d)^{ss}.$$
By the valuative criterion for properness, since $(\Mcal_d)^{ss}$ is projective, the composition of these morphisms, $\Ufrak_\varphi\rightarrow(\Mcal_d)^{ss}$ extends uniquely to C,
$$f':C\rightarrow(\Mcal_d)^{ss}.$$

Observe from our two constructions that $f'|_{U}=f|_{U}$.  The equality, $f'=f$, follows again from the valuative criterion.

Now we must note that the (scheme theoretic) points of $(\mathbb{P}^{2d+1})^{ss}$ that are not in $\mathrm{Rat}_d$ go to points of $(\Mcal_d)^{ss} \backslash\Mcal_d$ under $\pi$. This is because $\pi$ is the restriction to the semi-stable locus of the canonical rational map induced by the inclusion of graded rings $ \Z[\a,\b]^{SL_2} \hookrightarrow \Z[\a,\b]$.

With this in mind, as we stated earlier, $U_\Phi\subset\Ufrak_\varphi$.  Suppose $p\in\Ufrak_\varphi$ and $p$ not in $U_\Phi$.  Then $f(p)\in(\Mcal_d)^{ss} \backslash\Mcal_d $ by our construction of $f$ and the fact $p$ is not in $U_\Phi$.  On the other hand, by our construction of $f'$, $f'(p)\in\Mcal_d$ which is a contradiction.  Therefore, $\Ufrak_\varphi\backslash U_\Phi$ must be empty.
\end{proof}
An interesting question that we have not explored is \textit{when} a morphism has a semi-stable presentation.

The above theorem says that semi-stable presentations realize the conductor as their singular reduction locus, a property that any minimal presentation certainly must have. So one might hope that semi-stable presentations are also minimal.
\begin{conj} Let $\Phi$ be a semi-stable presentation of $\varphi$. Then $\Phi$ is a minimal presentation of $\varphi$.
\end{conj}

We have two partial results in this direction. One treats the degree two case: Theorems \ref{degreetwominimal} and \ref{ssminimaldegreetwo} below. The other is the following.
\begin{prop}
\label{semistableminimal}
Let $\Phi \in \ratd(K)$ be a semi-stable presentation.  Suppose $d = 2r$ is even. Let $p \in C$. Then $N_{\Phi, p}$ is minimal for all $\PGL_2(K)$ conjugates of $\Phi$ by diagonal elements $ [\Gamma] = \left[\begin{array}{cc}u & 0 \\0 & v\end{array}\right]$.
\end{prop}
\begin{proof}
We will show that
$$\frac{n_p(\a^{\Gamma}, \b^{\Gamma})}{v_p(\mathrm{det}(\Gamma))} \leq \frac{d+1}{2}$$
 for any $\left(\begin{array}{cc}u & 0 \\0 & v\end{array}\right)$ in $\mathrm{GL}_2(K)$ with $v_p(u)$ and $v_p(v)$ both greater than or equal to $0$. The result will follow by Proposition~\ref{minimalitycondition} (it is easy to see that this proposition can be appropriately restricted to any subset of $\PGL_2(K)$). By Theorem~\ref{numericalcriterion}, we know that at least one of the coefficients
\begin{align}\label{box0}
a_0, ... ,a_{r-1}, b_0, ... ,b_r
\end{align}
doesn't vanish at $p$. By conjugation, we may say the same about the coefficients,
$$a_r, ... , a_d, b_{r+1}, ... ,b_d.$$

Now conjugate by $\left(\begin{array}{cc}u & 0 \\0 & v\end{array}\right)$. The new coefficients are

\begin{align}\label{box1}
vu^da_0, v^2u^{d-1}a_1, ... , v^ru^{r+1}a_{r-1}, u^{d+1}b_0, vu^{d}b_1, ... , v^ru^{r+1}b_r
\end{align}
and
\begin{align}\label{box2}
v^{r+1}u^ra_r, v^{r+2}u^{r-1}, ... , v^{d+1}a_d, v^{r+1}u^rb_{r+1},v^{r+2}u^{r-1}b_{r+2} ... ,v^dub_d.
\end{align}

In (\ref{box1}), the power of $v$ that shows up is at most $r$, and in (\ref{box2}) the power of $u$ that shows up is at most $r$. Since the patterns of powers of $v$ that show up in (\ref{box1}) and (\ref{box2}) are identical when we permute $u$ and $v$ in one of them, it will be no loss of generality to suppose that $n = v_p(u) \leq v_p(v) = m$, and consider the coefficients in (\ref{box1}).

Now, $v_p(\mathrm{det}(\Gamma)) = n+m$. Take any of the coefficients in (\ref{box1}) such that the corresponding original coefficient from (\ref{box0}) doesn't vanish. Call this coefficient $c$. Thus the valuation of $c$ at $p$ will come entirely from the contributions of the powers of $u$ and $v$. Further, we know that the power $v$ makes at most a contribution of $rm$ to the valuation. Set $k \leq r$ to be the exponent on $v$ in $c$. We have
\begin{align*}
n_p(\a^{\Gamma}, \b^{\Gamma}) & \leq v_p(c) \\
&= (d+1-k)n +km \\
&= (d+1)n-kn+km \\
&= (d+1)n -2kn + kn +km \\
&= (d+1 -2k)n + k(n+m).\\
\end{align*}
Since $2k \leq 2r \leq d$, we have that $d+1 - 2k$ is positive.  Since $n \leq m$ implies that $n \leq \frac{m+n}{2}$, we have
\begin{align*}
(d+1 -2k)n + k(n+m) & \leq (d+1-2k) \left(\frac{m+n}{2}\right) + 2k\left(\frac{m+n}{2}\right) \\
&= (d+1)\left(\frac{m+n}{2}\right).
\end{align*}
Thus
$$\frac{n_p(\a^{\Gamma}, \b^{\Gamma})}{m+n} \leq \frac{d+1}{2}.$$
\end{proof}

The above argument does not work for $d$ odd. However, with Proposition~\ref{valuationformulas} (5), we may strengthen this result.
\begin{cor}
In the above setting, $N_{\Phi, p}$ is minimal among all conjugates of $\Phi$ of the form
$$[\Gamma] = [D][\Gamma'],$$
where $[D] = \left[\begin{array}{cc}u & 0 \\0 & v\end{array}\right]$ as above and $\Gamma' \in \mathrm{GL}_2(\mathcal{O}_{C,p})$.

\end{cor}
\begin{proof}
Proposition~\ref{valuationformulas} (5) implies directly that $v_p(\rho([\a^{D\Gamma'}, \b^{D\Gamma'}])) = v_p(\rho([\a^{D}, \b^{D}]))$, and then we apply Proposition~\ref{semistableminimal}.
\end{proof}


\subsection{Latt\`{e}s Maps}
\textit{Throughout this section, we will assume $K$ is a function field of a nonsingular curve over an algebraically closed field $k$ of characteristic zero}. We assume characteristic zero in order to avoid the problems that positive characteristic may cause in the calculations below. It is possible that this assumption could be relaxed somewhat by making appropriate assumptions about the residue characteristic, obtaining more general statements than those that follow.

Under these assumptions, we can construct $\Ratd \subset \PP_k^{2d+1}$ over the base field $k$. We assume this construction in place for this section only.

A family of maps that is useful to consider while studying dynamics in parallel to the theory of curves are the Latt\`{e}s maps. Interestingly, if $E(K)$ is an elliptic curve given by $y^2 = x^3+Ax+B$, then its discriminant is $4A^3+27B^2$. If $\varphi_E$ is the Latt\`{e}s map associated to the $x$ coordinate of the multiplication by 2 map on the elliptic curve, then the resultant of this dynamical system is, up to multiplication by an integer, the square of this discriminant. We will show this below.

Let $E$ be an elliptic curve over $K$. Denote by $[n]$ the multiplication by $n$ map with respect to the group structure on $E$. The quotient $E/\{\pm 1\}$ is isomorphic to $\pone_K$ and so we have a map:$$\pi : E \rightarrow \pone_K$$
The map $[n]$ descends to a map $\varphi_n$ on $\pone_K$ via this projection:
$$\xymatrix{ E \ar[r]^{[n]} \ar[d] & E \ar[d] \\ \pone_K \ar[r]^{\varphi_n} & \pone_K} $$
Now fix variables $X$, $Y$, and $Z$ giving a Weierstrass equation for $E$.  Let $x=X/Z$, $y = Y/Z$, and then $y^2 = x^3+Ax+B$ is the Weierstrass equation. Let $P(x) = x^3 + Ax + B$. With respect to the coordinates $[x,y,1]$, the map $\pi$ is then:
$$[x,y,1] \mapsto [x,1].$$
Further, with respect to the coordinate $[x,1]$ on $\pone_K$, $\varphi_n$ is then given explicitly as a rational function in terms of $x$, $A$, and $B$,
$$\varphi_n([x,1]) = [\frac{P(A, B, x)}{Q(A, B, x)}, 1].$$
This formula admits the form for n = 2,
$$\varphi_{2}([x,1]) = [\frac{(P'(x))^2-8xP(x)}{4P(x)},1].$$
From this we obtain the following proposition.
\begin{prop}
With the notation above, if $D$ is the discriminant of $E$ (which is the discriminant of $P(x)$) and $R$ is the resultant of $\varphi_{2}$, then $R = 256D^2$.
\end{prop}
\begin{proof}
This can be calculated directly using the formula for the resultant. Alternatively, we may observe the following.  The discriminant is itself the resultant
\begin{align}
\label{D}
D = \mathrm{Res}(P'(x), P(x)).
\end{align}
If one looks at the matrix whose determinant yields the resultant, it is clear that, in general, $\mathrm{Res}(f(x),g(x)) = \mathrm{Res}(f(x) +h(x)g(x), g(x))$ for any polynomials $f, g, h$. Thus we have
$$R=\mathrm{Res}((P'(x))^2-8xP(x), 4P(x)) = \mathrm{Res}((P'(x))^2, 4P(x)).$$
From the difference of roots formula for the resultant,
$$\mathrm{Res}((P'(x))^2, 4P(x)) = \mathrm{Res}(P'(x), 4P(x))^2.$$
Writing this as a product of differences of roots and factoring out a power of $4^2$ from the power of the leading coefficient gives
$$\mathrm{Res}(P'(x), 4P(x))^2 = (16\mathrm{Res}((P'(x)), P(x)))^2.$$
The result follows from (\ref{D}).
\end{proof}

An interesting question we have not yet explored is whether the Latt\`{e}s $[n]$ maps yield similar nice forms for the resultant in terms of the discriminant.

Semi-stablility in the parameter space of elliptic curves and semi-stablility in the parameter space of rational maps are related notions in that each is a special case of the more general GIT notion of semi-stability. Each can be shown to be equivalent to a condition involving the multiplicities on roots of polynomials. Specifically, curves that are given by Weierstrass equations with a double root, but not a triple root, are semi-stable in the GIT sense, with respect to the group action of change variables on the defining equation. And semi-stability of an element of $\Pk$, in the GIT sense, is equivalent to a condition involving a bound on the order of vanishing of common roots of the two polynomials; this is expressed in Theorem \ref{enhancednumericalcriterion} below. Asking whether semi-stable bad reduction of an elliptic curve over a function field implies that the reduced Latt\`{e}s map is semi-stable yields the following negative answer:

\begin{eg}
\label{unstable}
Let $E$ be an elliptic curve over the function field of a smooth curve $C$ with Weierstrass equation $y^2 = x^3 +Ax + B$.  Let $p \in C$ be a closed point such that evaluation of $A$ and $B$ at $p$ is defined, and such that the resulting reduced equation defines a singular curve (over $k$) with a node (i.e. semi-stable in the space of curves):
$$y^2 = x^3 + A(p)x + B(p) = (x-\lambda_1)(x-\lambda_2)^2\textrm{ with }\lambda_1\neq\lambda_2.$$
The associated Latt\`{e}s maps, $\varphi_n$ define points in $\mathbb{P}_k^{2n^2+1}$ which are never GIT semi-stable with respect to the action by $\PGL_2$.
\end{eg} 
\begin{proof}
Here, $n^2$ is the degree of the map over the function field.
We may now prove Example~\ref{unstable}. To ease notation, let $A(p) = a, B(p) = b$. Let $x^3 + ax + b = q(x)$. Write $\varphi_{n,p}$ for the reduced map at $p$. We first handle the case $n=2$. We have, by the representation of $\varphi_2$ above
$$\varphi_{2,p} = \frac{(q'(x))^2-8xq(x)}{4q(x)} = \frac{(2(x-\lambda_1)(x-\lambda_2) + (x- \lambda_2)^2)^2-8x(x-\lambda_1)(x-\lambda_2)^2}{4(x-\lambda_1)(x-\lambda_2)^2}.$$

One sees that $\lambda_2$ is at least a common double root of the numerator and denominator.  Since $\lambda_1 \neq \lambda_2$ by assumption, it is clear that it not a common triple root, and that canceling $(x-\lambda_2)^2$ gives us a rational map of degree 2. To apply the above proposition, part 1 (b), we must have that $\lambda_2$ is a fixed point of this map. This is a straightforward calculation. First note that, since the $x^2$ term of the Weierstrass equation is $0$, $\lambda_1 = -2\lambda_2$. Then by plugging in to the above and after cancelation, we have
$$\varphi_{2,p}(\lambda_2) = \frac{(2(\lambda_2+2\lambda_2))^2-8\lambda_2(\lambda_2+2\lambda_2)}{4(\lambda_2+2\lambda_2)}= 3\lambda_2-2\lambda_2 = \lambda_2.$$
Thus $\lambda_2$ is a fixed point, so $\varphi_{2,p}$ is unstable.

The case n = 2 is the only case that requires verification of this fixed point criterion. For $n \geq 3$, we will use Theorem~\ref{enhancednumericalcriterion}, parts 1 (a) and 2 (a). Following \cite{SilvermanEllipticCurves}, (p. 105 ex. 3.7), we can describe the $x$ coordinate of multiplication by $n$ in the group law. Given our elliptic curve with Weierstrass equation $y^2 = x^3 + Ax +B$, we define
\begin{eqnarray*}
\Psi_1 &=& 1\\
\Psi_2 &=& 2y\\
\Psi_3 &=& 3x^4 + 6Ax^2 + 12Bx - A^2\\
\Psi_4 &=& 4y(x^6 + 5Ax^4 + 20Bx^3 - 5A^2x^2 - 4ABx -8B^2 - A^3)
\end{eqnarray*}
and for $m \geq 2$
$$\Psi_{2m+1} = \Psi_{m+2}\Psi_m^3-\Psi_{m-1}\Psi_{m+1}^3$$
and
$$\Psi_{2m} = \frac{\Psi_m(\Psi_{m+2}\Psi_{m-1}^2 - \Psi_{m-2}\Psi_{m+1}^2)}{2y}.$$
Then $$\frac{x\Psi_n^2-\Psi_{n+1}\Psi_{n-1}}{\Psi_n^2} $$
is a rational function in $x$ of degree $n^2$ that gives us the $x$ coordinate of multiplication by $n$ in the group law.

For each $n \geq 1$, let $\psi_n$ be the polynomial in $k(x,y)$ obtained by evaluation of the coefficients $A$ and $B$ in $\Psi_n$ at $p$. Our goal is to see that the reduction
$$\varphi_{n,p} = \frac{x\psi_n^2-\psi_{n+1}\psi_{n-1}}{\psi_n^2} $$
has an appropriately high power of $(x-\lambda_2)$ dividing the numerator and denominator. For each $n$, let $M_n$ be $\mathrm{min}(v(x\psi_n^2-\psi_{n+1}\psi_{n-1}), v(\psi_n^2))$, where $v$ is the order of vanishing of $(x-\lambda_2)$. We are looking for a lower bound on $M_n$. The following table outlines the situation for the first few $n$. Here $d$ is the degree of the map $\varphi_n$, and $r$ is, in line with Theorem~\ref{enhancednumericalcriterion}, chosen so that $2r = d$ for $d$ even and $2r +1 = d$ for $d$ odd.

\bigskip

\begin{tabular}{c|c|c|c|c|c|c|c|c|c}$n$ & 2 & 3 & 4 & 5 & 6 & 7 & 8 & 9 & 10 \\\hline $d$ & 4 & 9 & 16 & 25 & 36 & 49 & 64 & 81 & 100 \\\hline $r$ & 2 & 4 & 8 & 12 & 18 & 24 & 32 & 40 & 50 \\\hline $r$+1 & 3 &  & 9 &  & 19 &  & 33 &  & 51 \\\hline $r$+2 &  & 6 &  & 14 &  & 26 &  & 42 & \end{tabular}

\bigskip

Separating the even and odd cases, the bottom two rows give us the lower bound on $M_n$ that we wish to show. When $n = 2m$ is even, the sequence of $r+1$ is given by $2m^2+1$. When $n=2m+1$ is odd, the sequence of $r+2$ is given by $2(m^2+m+1)$.

Again taking into account that $\lambda_1 = -2\lambda_2$, we have that $a= -3\lambda_2^2$ and $b= 2\lambda_2^3$. From this, one can compute
$$\psi_3 = 3(x-\lambda_2)^3(x+ 3\lambda_2)$$
$$\psi_4 = 4y(x-\lambda_2)^5(x+5\lambda_2).$$
Thus we may calculate
\begin{eqnarray*}
v(\psi_1) &=& 0 \\
v(\psi_2) &=& 1 \\
v(\psi_3) &=& 3 \\
v(\psi_4) &=& 6.
\end{eqnarray*}

This will allow us to inductively compute a lower bound on $2v(\psi_n)$ for each $n$. For reasons explained below, we will do this by hand for the next few $\psi$. Using these, we will then calculate a lower bound on the numerator as well, and thus on $M_n$ for the first few $n$.

\bigskip

\begin{tabular}{c|c|c|c|c}$n$ & lower bound on $2v(\psi_{n})$ & lower bound on $M_n$ & $r+1$ & $r+2$ \\\hline 2 & 2 & 2 & 3 &  \\\hline 3 & 6 & 6 &  & 6 \\\hline 4 & 12 & 12 & 9 &  \\\hline 5 & 18 & 18 &  & 14 \\\hline 6 & 26 & 26 & 19 &  \\\hline 7 & 36 & 36 &  & 26 \\\hline 8 & 48 & 48 & 42 &  \\\hline 9 & 60 & 60 &  & 51\end{tabular}

\bigskip

The above table establishes the result for $3 \leq n \leq 9$. It is computed as follows. First we write
$$ \psi_5 = \psi_4\psi_2^3 - \psi_1\psi_3^3$$
and using elementary properties of valuations, deduce that $v(\psi_5) \geq 9$. We then write $\psi_6$ in terms of the lower $\psi_n$, and use the lower bounds already established on them to make a similar deduction.  The rest follows.  This gives us the leftmost column of the table, and then the column next to it is a simple calculation of a lower bound on each $x\psi_n^2 + \psi_{n-1}\psi_{n+1}$. Note that these two columns are actually equal so far.  It may be the case that the lower bound is actually an equality and that the numerator and denominator always have the same power of $(x-\lambda_2)$ in them. This is not required for what we want to show, however.

\begin{lemma}
\label{lowerbound}
Suppose $m \geq 2$. Then:
\begin{quote}
{\bf (1)}  $2v(\psi_{2m}) \geq 2m^2 +1$ and\\
{\bf (2)}  $2v(\psi_{2m+1}) \geq 2(m^2 + m + 1).$
\end{quote}
\end{lemma}

The right hand side of these inequalities is $r+1$ and $r+2$ in the even and odd cases $n = 2m$ and $n= 2m+1$ respectively. The left hand side is $v$ of the denominator of $\varphi_{n,p}$. Thus Lemma \ref{lowerbound} will imply the result if we show that the numerator has at least as high a valuation as well. We now show that this also follows from Lemma \ref{lowerbound}. Our reasoning is valid for $m-1 \geq 2$ in the even case and $m \geq 2$ in the odd case, establishing the result for $n \geq 5$. The result for lower $n$ is given by the above table.

\vspace{.5cm}

\noindent \textbf{Even case} $(n=2m)$:
$$v(x\psi_{2m}^2- \psi_{2m+1}\psi_{2m-1}) \geq \mathrm{min}(2m^2+1, v( \psi_{2m+1}\psi_{2m-1}))$$
and we have
$$v( \psi_{2m+1}\psi_{2m-1}) \geq m^2 +m + 2 + (m-1)^2 + (m-1) + 1 = 2m^2+2$$
yielding a common root of order at least $2m^2+1 = r+1$.

\vspace{.5cm}

\noindent \textbf{Odd case} ($n=2m+1$):
$$v(x\psi_{2m+1}^2- \psi_{2m+2}\psi_{2m}) \geq \mathrm{min}(2v(\psi_{2m+1}), v(\psi{2m+2}) + v(\psi{2m}))$$
and we have
$$2v(\psi_{2m+1}) \geq 2(m^2 + m +1).$$
Then
$$ v(\psi_{2m+2}) + v(\psi_{2m}) \geq \frac{2(m+1)^2 + 1}{2} + \frac{2m^2 + 1}{2} = 2(m^2 + m +1)$$
yielding a common root of order at least $2(m^2+m+1) = r+2$.

\vspace{.5cm}

It remains to prove Lemma~\ref{lowerbound}. We must split into even and odd cases
\begin{eqnarray*}
m &=& 2s, \\
m &=& 2s + 1
\end{eqnarray*}
and verify a number of early cases by consulting the above table in order to get the induction going. We will be doing strong induction applied to values $2 \leq k \leq m$. In the even case we will need $s-1, s, s+1$ to be in this range. In the odd case, we will need $s, s+1, s+2$ to be in this range. Thus we will have to assume $m \geq 5$.  Again, the result for lower indices $m$ is given by the table above. To ease notation, we write $v = v$.

\vspace{.5cm}

Assume (1) and (2) hold for $2 \leq k \leq m$. We want to show

\bigskip

(a) $2v(\psi_{2(m+1)}) \geq 2(m+1)^2 +1$

(b) $2v(\psi_{2(m+1)+1}) \geq 2((m+1)^2 + (m+1) + 1).$

\bigskip

\textbf{Even case} ($m = 2s$)

For (a),
\begin{align*}
2v(\psi_{2m+1}) &= 2v(\frac{\psi_{m+1}(\psi_{m+3}\psi_{m}^2-\psi_{m-1}\psi_{m+2}^2)}{2y}) \\
&\geq 2(v(\psi_{m+1}) + min(v(\psi_{m+3}) + 2v(\psi_m), v(\psi_{m-1})+ 2v(\psi_{m+2})-1) \\
\end{align*}
and
\begin{align*}
2(v(\psi_{m+1}) + v(\psi_{m+3}) + 2v(\psi_m) -1)
&= 2(v(\psi_{2s+1}) + v(\psi_{2s+3}) + 2v(\psi_{2s})-1) \\
&\geq 2(s^2 + s + 1 + (s+1)^2\\ &+ (s+1) + 1 + 2s^2 + 1 - 1) \\
&= 2(4s^2 + 4s + 4)\\
&= 2(m^2 + 2m + 4) \\
&= 2(m+1)^2 + 6 \\
\end{align*}
while
\begin{align*}
2(v(\varphi_{m+1}) + v(\varphi_{m-1}) + 2v(\varphi_{m+2}) - 1) &= 2(v(\varphi_{2s+1}) + v(\varphi_{2s-1}) + 2v(\varphi_{2s+2}) - 1) \\
&\geq 2((s^2 + s + 1) + (s-1)^2\\ &+ (s-1) + 1 + 2(s+1)^2 + 1 - 1)\\
&= 2(m+1)^2 + 6.\\
\end{align*}
This proves (a).

For (b),
\begin{align*}
2v(\varphi_{2(m+1)+1)}) &= 2v(\varphi_{m+3}\varphi_{m+1}^3 - \varphi_{m}\varphi_{m+2}^3) \\
&\geq 2(min(v(\varphi_{m+3}) + 3v(\varphi_{m+1}), v(\varphi_{m}) + 3v(\varphi_{m+2})))
\end{align*}
and
\begin{align*}
2(v(\varphi_{m+3}) + 3v(\varphi_{m+1})) &= 2(v(\varphi_{2s+3}) + 3v(\varphi_{2s+1})) \\
&= 2(4s^2 + 6s + 6)\\
&= 2((m+1)^2 + (m+1) + 4)\\
\end{align*}
while
\begin{align*}
2(v(\varphi_{m}) + 3v(\varphi_{m+2})) &= 2(v(\varphi_{2s}) + 3v(\varphi_{2s+2})) \\
&\geq (2s^2 + 1 + 3(2(s+1)^2 + 1)) \\
&= 2((m+1)^2 + (m+1) + 3).\\
\end{align*}

\textbf{Odd case} ($m = 2s + 1$)

For (a),
\begin{align*}
2v(\psi_{2m+1}) &= 2v(\frac{\psi_{m+1}(\psi_{m+3}\psi_{m}^2-\psi_{m-1}\psi_{m+2}^2)}{2y}) \\
&\geq 2(v(\psi_{m+1}) + min(v(\psi_{m+3}) \\&+ 2v(\psi_m), v(\psi_{m-1})+ 2v(\psi_{m+2})-1) \\
\end{align*}
and
\begin{align*}
2(v(\psi_{m+1}) + v(\psi_{m+3}) + 2v(\psi_m))-1) &= 2(v(\psi_{2s+2}) + v(\psi_{2s+4}) + 2v(\psi_{2s+1}))-1) \\
&\geq (2(s+1)^2 + 1 + 2(s+2)^2 \\&+ 1 + 4(s^2 + s + 1) -2) \\
&= 2(4s^2 + 8s + 7) \\
&= 2(m+1)^2 + 6 \\
\end{align*}
while
\begin{align*}
2(v(\psi_{m+1}) + v(\psi_{m-1}) + 2v(\psi_{m+2}) -1) &= 2(v(\psi_{2s+2}) + v(\psi_{2s}) + 2v(\psi_{2s+3}) -1) \\
&\geq 2(4s^2 + 8s + 7) \\
&= 2(m+1)^2 + 6. \\
\end{align*}
Taking care of (a).

Finally, for (b),
\begin{align*}
2v(\psi_{2(m+1)+1}) &= 2v(\psi_{m+3}\psi_{m+1}^3 - \psi_{m}\psi_{m+2}^3)
\end{align*}
and
\begin{align*}
2v(\psi_{m+3}\psi_{m+1}^3) &= 2v(\psi_{2s+4}\psi_{2s+2}^3) \\
&= (2(s+2)^2 +1 + 3(2(s+1)^2 +1)\\
&= 8s^2 + 20s +17 \\
&= 2((m+1)^2 + (m+1) + 1) + 3
\end{align*}
while
\begin{align*}
2v(\psi_{m}\psi_{m+2}^3) &= 2v(\psi_{2s+1}\psi_{2s+3}^3) \\
&= 2(s^2 + s + 1 + 3((s+1)^2 + (s+1) + 1)) \\
&= 2((m+1)^2 + (m+1) + 4).
\end{align*}

\end{proof}

\section{Building a Counterexample to the Dynamical Analog to Theorem \ref{szpirostheorem}}
In this section, we will prove Example \ref{dynamicalszpirofalse}.

In what follows, by ``a counterexample to a dynamical analog to Theorem \ref{szpirostheorem},'' we mean an infinite collection of maps for which the degree of the conductor is bounded, and for which the degree of the minimal resultant is unbounded.
\subsection{Two Examples}
\subsubsection*{Example 1}

The following example is due originally to Patrick Ingram, who gave it in the number field case (with $t$ replaced by $p$ a prime number, below). It gives a collection of presentations of rational maps with only one point of singular reduction, but of unbounded degree of the resultant for that presentation. Since we do not know the presentation is minimal at the point of singular reduction, this does not give us an immediate counterexample to the dynamical analog of Theorem \ref{szpirostheorem}.

Let $C = \pone_k$, so $K = k(t)$. Consider the polynomial map
$$x^2 + t^{-N}$$
where $N > 0$. Written in projective coordinates, this is
$$\varphi = \frac{X^2 + t^{-N}Y^2}{Y^2}.$$
Let $\Phi = [\a, \b]$ be the associated presentation. We can see immediately that $\Phi$ has good reduction at $\infty$, so that $(R_{[\a,\b]})_{\infty} = 0$. For the reduction at the other points, we'll multiply throughout by $t^N$
$$\varphi =  \frac{t^NX^2 + Y^2}{t^NY^2}.$$
It is then a simple matrix calculation to see that the resultant of this is $t^{4N}$. This means that the only point of singular reduction with respect to this presentation is $t=0$. Notice here that the reduction is unstable, by Theorem~\ref{enhancednumericalcriterion} and $(R_{[\a,\b]})_0 = 4N$. Thus if we could show that this presentation is minimal for $t=0$, we'd have a family of rational maps with the degree of the conductor equal to $1$, and degree of the minimal resultant unbounded in terms of $N$.

In fact, this presentation is not minimal for $t=0$.  At least for $N = 2M$, there is a conjugate that lowers the degree of the resultant by a factor of 4. This is obtained by conjugating by $\left(\begin{array}{cc}1 & 0 \\0 & t^M\end{array}\right)$, yielding
$$\frac{t^{3M}X^2+t^{3M}Y^2}{t^{4M}Y^2} = \frac{X^2+Y^2}{t^MY^2}.$$

The resultant of this is computed to be $t^N$, so the degree of the resultant divisor at $t = 0$ has been lowered. However, it is still unbounded in terms of $N$.

\subsubsection*{Example 2}
\label{example2}
In analogy with hypotheses of Theorem \ref{szpirostheorem}, one might expect that having a semi-stable presentation would be a necessary or helpful condition for formulating and proving a dynamical analog. It is easy to see, however, that Example 1 has unstable reduction at $t=0$, so we don't know that this example qualifies. We are thus led to seek a counterexample to a dynamical analog of Theorem \ref{szpirostheorem}  that also admits a semi-stable presentation.

Suppose now $\varphi$ is a degree $2$ map that can be written in the form:
$$\frac{X^2 + \lambda_1XY}{\lambda_2XY + Y^2}.$$

Rational maps of this form are said to be in \textit{normal form}. In this form, the coefficients $\lambda_1$ and $\lambda_2$ are two of the multipliers of $\varphi$. For the presentations corresponding to such forms, we prove the following necessary and sufficient criterion for semi-stability.
\begin{prop}
\label{degreetwosemistable}
Let $\varphi$ be a morphism of degree 2 over $K$ that can be written in normal form:$$\varphi = \frac{X^2 + \lambda_1XY}{\lambda_2XY + Y^2}.$$ Then the corresponding presentation $\Phi =[1,\lambda_1, 0, 0, \lambda_2, 1]$ is a semi-stable presentation if and only if:
\begin{quote}
{\bf (1)}  any poles of $\lambda_1$ and $\lambda_2$ occur at exactly the same points, where moreover they have the same multiplicity and\\
{\bf (2)}  $\lambda_1$ and $\lambda_2$ never evaluate simultaneously to $1$.
\end{quote}
\end{prop}
\begin{proof}

By Theorem~\ref{enhancednumericalcriterion}, for singular reduction of a degree $2$ rational map to be unstable, it is necessary and sufficient that a common root showing up in reduction is either of order $2$, or of order $1$ while also being a fixed point of the map obtained after canceling all the common roots.

Suppose now $\Phi$ is a semi-stable presentation. Then if $\lambda_1$ has a pole of higher negative order than $\lambda_2$ at some $p \in C$, in order to take the reduction of $\Phi$ at $p$ we multiply the coefficients throughout by an appropriate power of the uniformizer to cancel this pole. Since the pole of $\lambda_1$ is of higher negative order, this causes the reduced map to look like
$$\frac{aXY}{0}.$$

Thus, when we cancel common factors, the resulting map is the constant map that sends everything to $\infty$. Thus $\infty$ is a fixed point of this map. Since $Y$ was one of the roots that we canceled, this map is unstable. A similar argument applies if $\lambda_2$ has a pole of higher negative order. This shows \textbf{(1)}.

For \textbf{(2)}, we see easily that, if $\lambda_1(p) = \lambda_2(p) = 1$, the reduced map looks like
$$\frac{X^2 + XY}{XY + Y^2}.$$
After canceling $X+Y$, we get
$$\frac{X}{Y}.$$
This is the identity map, and so it has $[-1,1]$, the canceled root, as a fixed point. Thus the reduction is unstable. This shows \textbf{(2)}.

Conversely, if \textbf{(1)} and \textbf{(2)} hold, first note that the resultant of $\Phi$ is $1-\lambda_1\lambda_2$, by a simple matrix calculation. Thus when the product $\lambda_1\lambda_2$ evaluates to $1$, we will have singular reduction. By \textbf{(2)}, the reduction must be of the form
$$\frac{X^2 + aXY}{a^{-1}XY + Y^2}$$
where $a \neq 1$. After cancelation, this is
$$
\frac{aX}{Y}$$
and $[-a, 1]$, the canceled root, is not a fixed point. Hence the reduction is semi-stable.

The other possibility for singular reduction is when some of the coefficients have poles, which in this case, by \textbf{(1)}, can only happen at the common poles of $\lambda_1$ and $\lambda_2$. For these, we see that the reduction is of the form
$$\frac{aXY}{bXY}$$
where $a$ and $b$ are not $0$. Thus after canceling the common roots $X$ and $Y$, we get a constant map that is not $0$ or $\infty$, so that neither $0$ nor $\infty$ is a fixed point, and hence the reduction is semi-stable.
\end{proof}
\begin{cor}
\label{counterexample0}
Suppose $a, b, b' \neq 1$ and let, for each $N \in \Z$ with $N \geq 1$
\begin{eqnarray*}
\lambda_1 &=& a + bt^N \\
\lambda_2 &=& a^{-1} +b't^N.
\end{eqnarray*}
where $a, b, b' \in k$ satisfy $a, b, b' \neq 0, 1$, and $ab'+b/a = 0$. Then $\Phi = [1, \lambda_1, 0, 0, \lambda_2, 1]$ is a semi-stable presentation. In addition, for each $N$, the non-singular reduction locus of $\Phi$ contains 2 points, and the degree of the resultant divisor of $\Phi$ is at least $2N$.
\end{cor}
\begin{proof}
The corresponding $\Phi$, by construction, satisfies \textbf{(1)} and \textbf{(2)} and is therefore a semi-stable presentation by Proposition \ref{degreetwosemistable}. The only pole of $\lambda_1, \lambda_2$ is $\infty$, and so we have one point of singular reduction there. The resultant is
 \begin{align*}
 1-\lambda_1\lambda_2 &= 1-(a+bt^N)(a^{-1} + b't^N) \\
 &= -(ab'+b/a)t^N -bb't^{2N} \\
 &= -bb't^{2N}.
 \end{align*}
Thus we have only one other point of singular reduction (at $t=0$), whose multiplicity in the resultant divisor for this presentation is $2N$. Thus the non-singular reduction locus has degree 2, and the degree of the resultant divisor is unbounded in terms of $N$.
\end{proof}

Corollary \ref{counterexample0} implies that the degree of the resultant divisor of a semi-stable presentation cannot be bounded in terms of the conductor. In order for this to be a counterexample to a dynamical analog of Theorem \ref{szpirostheorem}, there is still the question of the minimality of this presentation at the unbounded point.

\subsection{Proving Minimality}

In attempting to show the minimality of these two examples at the unbounded points, one line of attack is to act by the group and then analyze the order of vanishing at $t$ of the coefficients, and apply Proposition~\ref{minimalitycondition}. Doing this for Example 2 leads to the following partial result.

\begin{prop}
The above example is minimal at $t=0$ with respect to conjugations of the form $\left(\begin{array}{cc}\alpha & \beta \\0 & 1\end{array}\right)$ where $v_0(\alpha) > 0$.
\end{prop}
\begin{proof}
We apply the criteria of Proposition~\ref{minimalitycondition}, but with respect to the restricted subset of of $\mathrm{GL}_2(K)$ stated in the theorem. First, by the inner substitution
\begin{align*}
\frac{X^2 + \lambda_1XY}{\lambda_2XY + Y^2} \rightarrow &  \frac{(\alpha X+\beta Y)^2 + \lambda_2(\alpha X+\beta Y)Y}{\lambda_2(\alpha X+\beta Y)Y + Y^2} \\
&= \frac{\alpha^2X^2 + (2\alpha\beta + \lambda_1\alpha)XY + (\beta^2 + \lambda_1 \beta)Y^2}{\lambda_2 \alpha XY +(\lambda_2 \beta +1)Y^2}.
\end{align*}
Now performing the outer substitution yields
\begin{align}
\label{conjugated}
\frac{\alpha^2X^2 + (2\alpha\beta + \lambda_1 \alpha - \beta\lambda_2\alpha)XY + (\beta^2 + \lambda_1\beta - \lambda_2\beta^2 - \beta)Y^2}{\lambda_2\alpha^2XY + (\alpha\lambda_2\beta + \alpha)Y^2}.
\end{align}
Suppose minimality fails. This means we can find $\alpha,\beta$ such that minimal valuation at $0$ of the coefficients of (\ref{conjugated}), which we will call $n$, satisfies
$$\frac{n}{v_0(\alpha)} > \frac{3}{2}.$$
This will certainly imply
$$n > v_0(\alpha).$$
This imposes three restrictions
\begin{quote}
{\bf (a)}  $v_0(\alpha\lambda_2\beta + \alpha) > v_0(\alpha),$\\
{\bf (b)}  $v_0(2\alpha\beta + \lambda_1\alpha -\beta\lambda_2\alpha) > v_0(\alpha),$\\
{\bf (c)}  $v_0(\beta(\beta+ \lambda_1-\lambda_2\beta -1)) > v_0(\alpha).$
\end{quote}
Factoring $\alpha$ in (a) and (b), and observing that by assumption $v_0(\alpha) > 0$ in 3, we reduce these to
\begin{quote}
{\bf (a)}  $v_0(\lambda_2\beta + 1) > 0,$\\
{\bf (b)}  $v_0(2\beta + \lambda_1 -\beta\lambda_2) > 0,$\\
{\bf (c)}  $v_0(\beta(\beta+ \lambda_1-\lambda_2\beta -1)) > 0.$
\end{quote}
Now, since $\lambda_2$ doesn't vanish at $0$, (a) implies that $v_0(\beta) = 0$, and so we already have a contradiction if $v_0(\beta) \neq 0$, proving minimality among those. If $v_0(\beta) = 0$ then (c) reduces to
\begin{align}
\label{A}
v_0(\beta + \lambda_1 - \lambda_2\beta -1) > 0.
\end{align}
Let $\beta_0$ be the evaluation of $\beta$ at 0. Then it follows from (\ref{A}) and (b) respectively that
\begin{align}
\label{B}
\beta_0 + a -a^{-1}\beta_0 -1 = 0
&\\ 2\beta_0 + a -a^{-1}\beta_0 = 0.
\end{align}
Subtracting yields
$$-\beta_0-1 = 0,$$
so $\beta_0 = -1$. But then by (\ref{B}), $-2+a+a^{-1} = 0$, which implies $a=1$, a contradiction.
\end{proof}

The brute force methods used in the above proof--using the basic properties of valuations in an effort to find a contradiction--appear to be insufficient to test for minimality in general. But a different approach may be taken. We are able to show the minimality of the above example via a minimality criteria that is proven using the $PGL_2$ invariance of the symmetric functions of the periodic points of the multipliers of a rational map.

For a given $\varphi$, consider the periodic points of a fixed period $n$. Let $\sigma_{n,i}(\varphi)$ be the $i$-th symmetric function in the multipliers of these periodic points. Each $\sigma_{n,i}(\varphi)$ is $PGL_2$ invariant and thus depends only on $\varphi$ and not on the presentation and model we choose to represent it. In  \cite{SilvermanDynamics}, it is shown further that each $\sigma_{n,i}(\varphi)$ can be written in terms of the coefficients of any model $(\a,\b)$ of a presentation $\Phi$ of $\varphi$ via an expression of the following form:
\begin{align}
\label{sigmain}
\sigma_{n,i} &= \frac{P(\a,\b)}{(\rho(\a,\b))^m}
\end{align}
where $m \geq 0$ is an integer and $P(\a,\b)$ is a homogeneous polynomial of degree $2dm$ (so that the fraction is degree zero).

\begin{prop}
\label{DegreeTwoMinimalCriterion}
Let $\Phi = [\a, \b]$ be a presentation of a degree $d$ rational map $\varphi$ over a field $K$. Let $p \in C$ be a point of singular reduction, and suppose that $(\a,\b)$ is a $p$-model of $\Phi$. Let $P(\a,\b)$ and $\sigma_{n,i}(\varphi)$ be as above, and suppose $m \geq 1$. If $v_p(P(\a,\b))= 0$, then $(R_{\varphi})_p = (R_{\Phi})_p$ (i.e. $R_{\Phi}$ is minimal at $p$).
\end{prop}
\begin{proof}
Recall that, in general, $(R_{\Phi})_p$ is simply the valuation of $\rho(u_p\a,u_p\b) \in K$, where $u_p$ is chosen so that $(u_p\a,u_p\b)$ is a $p$-model of $\Phi$. Since we have assumed this already holds, we can take $u_p = 1$, i.e.
\begin{align}
\label{S}
S := (R_{\Phi})_p =v_p(\rho(\a,\b)).
\end{align}

Since $(\a,\b)$ is a $p$ model, all of the coefficients of $(\a,\b)$ are in $\OCp$. Let $[\Gamma] \in PGL_2(K)$, where the representative $\Gamma = \left(\begin{array}{cc}\alpha & \beta \\\gamma & \delta\end{array}\right)$ is a $p$-model of $[\Gamma]$. Let $(\a'', \b'')$ be the new coefficients obtained via the action of $\Gamma$ on $(\a,\b)$. Viewed on the affine cone, this action is pre-composition by $\Gamma$ and post-composition by the adjoint of $\Gamma$ (it descends to the conjugation action when we pass to projective space). Because of this, the new coefficients $(\a'', \b'')$ are in $\OCp$.

We now cancel the greatest common power of the uniformizer occurring in each coefficient; this gives us new coefficients $(\a',\b')$, still all in $\OCp$, and now $(\a',\b')$ is a $p$-model of $[\a^{\Gamma},\b^{\Gamma}]$. Now we have that

\begin{align}
\label{S'}
S' :=(R_{\Phi^{\Gamma}})_p = v_p(\rho(\a',\b')).
\end{align}
We need to show that $S \leq S'$. Now, by assumption, $v_p(P(\a,\b)) = 0$. This implies, by (\ref{sigmain}) and (\ref{S}), that  $v_p(\sigma_1(\varphi)) = -Sm$. Let $r = v_p(P(\a',\b'))$. Since the coefficients $(\a',\b')$ are in $\OCp$, we have that $r \geq 0$. Thus by (\ref{sigmain}), (\ref{S'}), and $PGL_2(K)$ invariance:
\begin{align*}
-Sm &= v_p(\sigma_1(\varphi^{\Gamma})) \\
&= v_p\left(\frac{P(\a',\b')}{\rho(\a',\b')} \right)\\
&= r-S'm \\
&\geq -S'm
\end{align*}
Hence $S \leq S'$.
\end{proof}

\begin{cor}
\label{counterexample}
Let $K = k(t)$. Let $\varphi$ be the degree 2 morphism given by: $$\frac{X^2 + \lambda_1XY}{\lambda_2XY + Y^2}$$ where $\lambda_1 = a+bt^N$, $\lambda_2 = a^{-1} + b't^N$, $a, b, b' \neq 0,1$, and $ab'+b/a = 0$. Let $\Phi$ be the corresponding presentation. Then $(R_{\Phi})_0 = (R_{\varphi})_0$ (i.e. $R_{\Phi}$ is minimal at $ 0 \in \mathbb{A}^1$).
\end{cor}
\begin{proof}
The symmetric function $\sigma_1(\varphi)$ is simply $$\lambda_{P_1} + \lambda_{P_2} +\lambda_{P_3},$$
where the $\lambda_{P_i}$ are the multipliers of the fixed points $P_i$ of $\varphi$. Silverman provides, in \cite{SilvermanDynamics}, the coefficients for the general formula for $\sigma_1$ mentioned above, and this makes it easy for us to calculate $\sigma_1(\varphi)$ (alternatively, we could make the calculation as in Theorem \ref{degreetwominimal} below). The formula is as follows. First set

\begin{eqnarray*}
P(\a,\b) = & a_1^3b_0 - 4a_0a_1a_2b_0-6a_2^2b_0^2-a_0a_1^2b_1+4a_1a_2b_0b_1 \\ & -2a_0a_2b_1^2+a_2b_1^3-2a_1^2b_0b_2 + 4a_0a_2b_0b_2 \\ & -4a_2b_0b_1b_2 -a_1b_1^2b_2+2a_0^2b_2^2+4a_1b_0b_2^2
\end{eqnarray*}

Then we have

\begin{align}
\label{sigma1}
\sigma_1(\varphi) = \frac{P(\a,\b)}{a_2^2b_0^2-a_1a_2b_0b_1+a_0a_2b_1^2 + a_1^2b_0b_2 - 2a_0a_2b_0b_2-a_0a_1b_1b_2+a_0^2b_2^2}.
\end{align}

Since the given model is a $p$-model for $p=0 \in \mathbb{A}^1$, all we must show is that the form $P(\a,\b)$ doesn't vanish at $t=0$ for these coefficients. Thus we must show that the constant term is nonzero; this is a simple calculation by plugging in to the expression for $P(\a, \b)$ above. Most of the terms vanish, and we are left with the expression:

$$-(a+bt^N)^2(a^{-1} + b't^N) - (a+bt^N)(a^{-1}+b't^N)^2 + 2.$$

The constant term of this expression is $-a-a^{-1}+2$, which can never be zero, by the assumption that $a \neq 1$.
\end{proof}

By generalizing the above proof, we can show that semi-stable presentations of degree two maps in normal form are minimal at certain places.
\begin{thm}
\label{degreetwominimal}
Let $\Phi$ be a normal form presentation of a degree $2$ morphism $\varphi$: $$\varphi = \frac{X^2 + \lambda_1XY}{\lambda_2XY + Y^2}.$$ Suppose $\Phi$ has everywhere semi-stable reduction, and let $P$ be the common poles of $\lambda_1$ and $\lambda_2$, as described in Proposition \ref{degreetwosemistable}. Then $(1,\lambda_1, 0, 0, \lambda_2, 1)$ is a $P$-minimal global model for $\Phi$.
\end{thm}
\begin{proof}
The singular reduction of $\Phi$ occurs either at the common poles of $\lambda_1$ and $\lambda_2$, or where $\lambda_1\lambda_2$ evaluates to $1$. We must show that in the latter case $(R_{\Phi})_p$ is minimal. Let $p$ be such a point of singular reduction. It is known (see \cite{SilvermanDynamics}) that in general $\sigma_1 = \sigma_3 + 2$. Thus if we let $\lambda_3$ be the third multiplier of $\varphi$ (in normal form, the first two multipliers are $\lambda_1$ and $\lambda_2$), we have:
$$\sigma_1 = \lambda_1+\lambda_2+\lambda_3 = \lambda_1\lambda_2\lambda_3 +2. $$

Thus $$\lambda_3 = \frac{2-\lambda_1-\lambda_2}{1 - \lambda_1\lambda_2}$$

so that $$\sigma_1 = \lambda_1 + \lambda_2 +  \frac{2-\lambda_1-\lambda_2}{1 - \lambda_1\lambda_2}.$$

From this we can calculate $P(\a,\b)$ directly in terms of $\lambda_1$ and $\lambda_2$:

\begin{align*}
\mathrm{Res}(\Phi)\sigma_1 &= (1-\lambda_1\lambda_2)(\lambda_1 + \lambda_2 +  \frac{2-\lambda_1-\lambda_2}{1 - \lambda_1\lambda_2}) \\
&= -\lambda_1^2\lambda_2 - \lambda_1\lambda_2^2 + 2.
\end{align*}

Let $a = \lambda_1(p)$. Then since the resultant vanishes at $p$, $\lambda_2(p) = a^{-1}$. Further, we see from Proposition \ref{degreetwosemistable} that $a \neq 1$. Hence we have:
$$-\lambda_1(p)^2\lambda_2(p) - \lambda_1(p)\lambda_2(p)^2 + 2 = -a-a^{-1} +2 \neq 0.$$

By Proposition \ref{DegreeTwoMinimalCriterion}, $R_{\Phi}$ is minimal at $p$.
\end{proof}

Using a slightly different approach, we now show that, in the function field situation, semi-stability implies minimality for all degree two maps.
\begin{thm}
\label{ssminimaldegreetwo}
Let $K$ be a function field over an algebraically closed field $k$. Let $\varphi$ be a morphism of degree two, and let $\Phi$ be a presentation of $\varphi$. Let $p \in C$. If the reduction of $\Phi$ at $p$ is semi-stable, then $(R_{\varphi})_p = (R_{\Phi})_p$. 

\end{thm}
\begin{proof}
First we need a lemma.

\begin{lemma}
\label{reductioncommutes}
Let $K$ a function field over an algebraically closed field $k$. Let $\varphi$ be a morphism, $\Phi \in \Ratd(K)$ a presentation of $\varphi$, and $p \in C$. Let $\Gamma \in GL_2(k)$. Then $(\a^{\Gamma},\b^{\Gamma})$ is a $p$-model of $\Phi^{\Gamma}$ and $(\Phi^{\Gamma})_p = \Phi_p^{\Gamma}$.
\end{lemma}

\begin{proof}
Let $(\a,\b)$ be a $p$-model of $\Phi$. Each coefficient of  $(\a^{\Gamma}, \b^{\Gamma})$ is just a polynomial in $\alpha, \beta, \gamma, \delta, a_0, \cdots, a_d, b_o, \cdots, b_d$, and so it is clear that $(\a^{\Gamma}(p), \b^{\Gamma}(p)) = (\a(p)^{\Gamma}, \b(p)^{\Gamma})$. Since the right hand side descends to a point of projective space, $(\a^{\Gamma}, \b^{\Gamma})$ must be a $p$-model of $\Phi^{\Gamma}$, and we have the desired equality.
\end{proof}
Let now $(\a,\b)$ be a $p$-model of our $\Phi$. It follows from Lemma 6.2 of \cite{SilvermanSpace} that we may find a $\Gamma \in GL_2(k)$ such that $\Phi_p^{\Gamma} = [0,A,0,0,1,B]$ where $A, B$ are not both zero. Since $\Gamma \in GL_2(\OCp)$, it follows from Proposition \ref{valuationformulas} (5) that it suffices to show minimality of $\Phi^{\Gamma}$.  

We have formulas for $\rho(\a,\b)\sigma_1(\varphi)$ and $\rho(\a,\b)\sigma_2(\varphi)$ (the former is given above and both appear together on p.17 of \cite{SilvermanSpace}) and they are polynomials in the coefficients $(\a,\b)$. Applying these formulas to the coefficients $(\a^{\Gamma}, \b^{\Gamma})$, we may show the minimality of $\Phi^{\Gamma}$ at $p$ by showing that at least one does not evaluate to zero at $p$, by Proposition \ref{DegreeTwoMinimalCriterion}.

By Lemma \ref{reductioncommutes}, we know that $(\a^{\Gamma}(p), \b^{\Gamma}(p)) = (0, A, 0, 0, 1, B)$. Plugging in thus yields: \begin{enumerate}

\item $\rho(\a^{\Gamma},\b^{\Gamma})\sigma_1(\varphi)(p) = -AB$
\item $\rho(\a^{\Gamma},\b^{\Gamma})\sigma_2(\varphi)(p) = -A^2-B^2.$
\end{enumerate}
Clearly (1), (2) cannot simultaneously be zero, and so we are done.
\end{proof}
The above argument relies on the fact that the residue field of $p$ is embedded in $\OCp$, which fails, for example, in the number field situation.

\subsection{The Critical Conductor}

In \cite{Szpiro4}, the first author and T. Tucker propose an alternative definition of good reduction, called \textbf{critical good reduction}, which is further studied in \cite{CPT}. Here we show that it is at least possible that the complement of the locus of the critical good reduction increases without bound for Example 2 above, implying that this example is not necessarily a counterexample to the dynamical analog of Theorem \ref{szpirostheorem} under this alternative definition of the conductor.  Note that C. Petsche, in \cite{Petsche}, gives another definition of conductor on a modified space.

Following \cite{Szpiro4}, Let $\varphi$ be a morphism of degree at least 2 over $K$.  Let $[X,Y]$ be coordinates on $\pone_{K}$. Let $R(\varphi)$ be the ramification divisor of $\varphi$ over the algebraic closure $\bar{K}$ of $K$, and let $K'$ be an extension of $K$ such that the points in the support of the ramification divisor are in $\mathbb{P}^1_{K'}$. Let $C'$ be the corresponding curve. Then for $P \in \pone_{K'}$, we can define the reduction of $P$ at $p \in C'$ by choosing coefficients $(\alpha, \beta) \in \mathcal{O}_{C',p}^2$ with at least one coefficient a unit, such that $[\alpha,\beta] = P$, and then evaluating these coefficients: $$r_p(P) := [\alpha(p), \beta(p)].$$
\begin{definition} With the notations and assumptions as above, $\varphi$ has \textbf{critical good reduction} at $p \in C$ if:
\begin{enumerate}
\item If $P, Q \in \mathrm{Supp}(R(\varphi))$ and $P \neq Q$, then $r_p(P) \neq r_p(Q)$.
\item If $P, Q \in \varphi(\mathrm{Supp}(R(\varphi)))$ and $P \neq Q$, then $r_p(P) \neq r_p(Q)$.
\end{enumerate}
If these conditions are not satisfied, then $\varphi$ has \textbf{critical bad reduction} at $p$.
\end{definition}
\begin{definition} \label{criticalconductor}
The \textbf{critical conductor} of $\varphi$ is the divisor $$\mathfrak{f}_{cr} = \sum_{p \in S} [p]$$ where $S$ is the set of critical bad reduction.
\end{definition}
This definition depends on the choice of coordinates on $\pone_{K}$. What we will now show is that, with respect to the coordinates $[X,Y]$ in Example 2 above, $\varphi$ has many (unbounded in terms of $N$) points of critical bad reduction.
\begin{prop}
\label{critlarge}
Let $\varphi$ be given by: $$\varphi = \frac{X^2 + \lambda_1XY}{\lambda_2XY + Y^2}$$ where $\lambda_1 = a+bt^N$, $\lambda_2 = a^{-1} + b't^N$, $a, b, b' \neq 0,1$, and $ab'+b/a = 0$. Then, with respect to the coordinates $[X,Y]$, $\varphi$ has at least $N + 1$ points of critical bad reduction.
\end{prop}
\begin{proof}
A simple derivative calculation shows that the ramification divisor is defined over $K$ and consists of exactly two points:
$$P= \frac{-1+i\sqrt{bb'}t^N}{a^{-1}+b't^N}, Q=\frac{-1-i\sqrt{bb'}t^N}{a^{-1}+b't^N}.$$
$P$ and $Q$ will specialize to the same point over $\pone_k$ when $t = 0$, and also when the (common) denominator of these two fractions evaluates to $0$. Since $a^{-1}$ and $b'$ are not zero, the equation $a^{-1}+b't^N$ has $N$ distinct roots, yielding an additional $N$ points of critical bad reduction.
\end{proof}

Combining this with Example \ref{dynamicalszpirofalse}, we have:

\begin{eg}
\label{alltogether}
For each $N \in \Z^{+}$, let $\varphi$ be the degree 2 morphism given by: $$\varphi = \frac{X^2 + \lambda_1XY}{\lambda_2XY + Y^2}$$ where $\lambda_1 = a+bt^N$, $\lambda_2 = a^{-1} + b't^N$, $a, b, b' \neq 0,1$, and $ab'+b/a = 0$. Then the degree of the conductor of $\varphi$ is at most 2, and the degree of the minimal resultant is $2N$, and the degree of the critical conductor is at least N+1.
\end{eg}

\subsection{The Number Field Case} In the number field setting, the standard notion of the degree of a divisor is the norm of the associated ideal, made additive by applying $\log$. Thus the degree of the prime divisor $\frak{p}$ on $\mathrm{Spec}(O_K)$ is $f\log p$ where $f$ is the residue degree of $\frak{p}$ over the prime integer $p$. 

Examples 1 and 2 above can be given instead for number fields if we replace the variable $t$ by a prime $p \in \Z$. It is then possible to construct a similar counter example, in the number field setting, to the dynamical analog of theorem \ref{szpirostheorem}.

\begin{eg}
\label{counterexample0numberfield}
Let $K$ be a number field of degree $n$ over $\Q$, let $p \in \Z$ be prime, and let $\{\frak{p}_i\}$ be the primes of $\mathcal{O}_K$ lying over $p$. Suppose $a, b, b'$ are units of $\mathcal{O}_K$ and let, for each $N \in \Z$ with $N \geq 1$
\begin{eqnarray*}
\lambda_1 &=& a + bp^N \\
\lambda_2 &=& a^{-1} +b'p^N.
\end{eqnarray*}
where $a, b, b' \in K$ also satisfy $a, b, b' \neq 0, 1$, and $ab'+b/a = 0$. Then for each $N$, the degree of the conductor is at most $n(\mathrm{log}p)$, and the degree of the minimal resultant divisor is $2nN \log p$.
\end{eg}
\begin{proof}
Let $e_i$ and $f_i$ be the ramification indexes and residue degrees of the $\frak{p}_i$ respectively. We will use the basic fact that \begin{align}\label{ef}\sum_i e_if_i =n.
\end{align}

The presentation given is a $\frak{q}$ model for every $\frak{q} \in \mathrm{Spec}(\mathcal{O}_K)$. The resultant for the given presentation, as before, is $bb'p^N$, and since $b$ and $b'$ are units, we therefore have the singular reduction occurring exactly at the primes $\frak{p}_i$ dividing $p$. The degree of the conductor is therefore $\sum_i f_i \log p$ (the $f_i$ being the residue degrees of the $\frak{p}_i$), which is at most $n(\mathrm{log}p)$ by (\ref{ef}). For each $i$, $v_{\frak{p}_i}(bb'p^N) = 2Ne_i$, which, by (\ref{ef}) again, gives the desired degree for the resultant divisor of this presentation. 

It follows from Theorem \ref{degreetwominimal} and the assumption that $a \neq 1$ that the given presentation is minimal at each $\frak{p}_i$, and so we are done.
\end{proof}

A concrete example realizing the conditions of Example \ref{counterexample0numberfield} is the following: Let $K =Q(\zeta)$, where $\zeta$ is a primitive $m$-th root of unity for some $m$. Then take $a=\zeta, b = -\zeta,$ and $b'=\zeta^{m-1}$. 

The question of what happens to the critical bad reduction for this example, in the number field case, is interesting. It  turns out  that if we assume the following conjectural version of Theorem \ref{szpirostheorem} for number fields we can show that the critical bad reduction for this example grows with $N$.

\begin{conj} (Discriminant Conjecture) Let E be an elliptic curve over a number field K of discriminant $D(K)$, let  $\Delta (E)$ be the norm of the minimal discriminant of E, and $\mathfrak{f}(E)$ the norm of its conductor. Then for every  $\epsilon > 0$ one has : $$ \Delta(E) \leqslant _\epsilon (D(K) \mathfrak{f}(E))^{6+\epsilon}.$$
 \end{conj}
This conjecture of the first author is closely related to the ABC conjecture when  applied to elliptic curves of the form: $$y^2=x(x-A)(x+B)$$(cf. \cite{Goldfeld}).
 
 \begin{prop}
 \label{critnumberfield}
 Under the conditions of Example \ref{counterexample0numberfield}, the discriminant conjecture implies that the norm of the critical conductor is bounded below by a linear increasing function of $p^N$.
 \end{prop}
  
 \begin{proof}
 Let $\epsilon>0$, and let $C_\epsilon$ be the implied constant in the discriminant conjecture. The same calculation as given in Proposition \ref{critlarge} shows the critical bad reduction occurs at the usual bad reduction of the given presentation $\Phi$ of $\varphi$, and also where $\lambda_2= a^{-1}+b'p^N$ vanishes.  The elliptic curve with equation:$$y^2=x(x-a^{-1})( x+b'p^N)$$ has minimal discriminant  $2^{-8} a^{-2} b'^2 p^{2N}$. The discriminant  conjecture gives$$ \mathrm{Norm}(\mathfrak{f}_{cr})^{6+\epsilon} = \mathrm{Norm}((\lambda_2 p))^{6+\epsilon} \geqslant C_\epsilon' p^{2N}.$$
 \end{proof}
 \section{Conjectures}
In the spirit of passing from elliptic curves, and their Latt\`{e}s maps, to dynamical systems on the sphere, we present two conjectures.  The result of T. Tucker and the first author in \cite{Szpiro4} suggests that fixing the critical conductor leads to boundedness results.
\begin{conj} (Resultant conjecture for a function field) If K is a function field (over a field k), and $\varphi$ is a self map of degree d of $\mathbb{P}^1_K$ with minimal resultant $R(\varphi)$ and critical conductor $\mathfrak{f}_{cr}(\varphi)$, then there exists a constant $C(K)$ and integer $s(d)$ such that $$\mathrm{deg}(R(\varphi))\leqslant p^e s(d)(C(K)+\mathrm{deg}(\mathfrak{f}_{cr}(\varphi)))$$
where $p$ is the characteristic of $k$ and $e$ is the inseperability degree of the canonical map, induced by $\varphi$, from $\mathrm{Spec}(K)$ to the moduli space $\mathcal{M}_d$. 
 \end{conj}
 
 \begin{conj} (Resultant conjecture for a number field) If K is a number field, and $\varphi$ is a self map of degree d of $\mathbb{P}^1_K$ with minimal resultant of norm $R(\varphi)$ and critical conductor of norm $\mathfrak{f}_{cr}(\varphi)$, then there exists a constant $C(K)$ and integer $s(d)$ such that for every $\epsilon > 0$: $$(R(\varphi))\leqslant_{\epsilon} (C(K)\mathfrak{f}_{cr}(\varphi))^{s(d)+\epsilon}.$$
 \end{conj}

\newpage
 \end{document}